\documentclass[11pt]{article}

\usepackage[margin=1truein]{geometry}
\usepackage{authblk}

\newcommand{\myQED}{}

\usepackage{graphicx}
\usepackage{mathptmx}

\usepackage{ascmac}
\usepackage{appendix}
\usepackage[linesnumbered, ruled]{algorithm2e}
\usepackage{amsmath}
\usepackage{amssymb}
\usepackage{amsthm}
\usepackage{bm,bbm}
\usepackage{booktabs}
\usepackage{cancel}
\usepackage{enumitem}
\usepackage{fancybox}
\usepackage{floatrow}
\usepackage[breaklinks]{hyperref}
\usepackage{ifthen}
\usepackage{lipsum}
\usepackage{makecell}
\usepackage{mathrsfs}
\usepackage[showonlyrefs]{mathtools}
\usepackage{multirow}
\usepackage{optidef}
\usepackage{orcidlink}
\usepackage{physics}
\usepackage{pifont}
\usepackage{placeins}
\usepackage{subcaption}
\usepackage{subfiles}
\usepackage{thm-restate}
\usepackage{tikz}
\usepackage{xparse}

\makeatletter
\let\cl@chapter\undefined
\makeatletter

\usepackage{cleveref}
\usepackage[numbers,sort&compress]{natbib}

\hypersetup{colorlinks=true,linkcolor=blue,citecolor=blue,urlcolor=blue}

\definecolor{cA}{HTML}{0072BD}
\definecolor{cB}{HTML}{EDB120}
\definecolor{cC}{HTML}{77AC30}
\definecolor{cD}{HTML}{D95319}

\newenvironment{MPCchange}{}{}
\newcommand{\MPCchanged}[1]{#1}

\newcommand{\defeq}{\coloneqq}

\newcommand{\relmiddle}[1]{\mathrel{}\middle#1\mathrel{}}

\ExplSyntaxOn
\NewDocumentCommand{\definealphabet}{mmmm}{
    \int_step_inline:nnn{`#3}{`#4}{
        \cs_new_protected:cpx{#1 \char_generate:nn{##1}{11}}{
            \exp_not:N #2{\char_generate:nn{##1}{11}}}}}
\ExplSyntaxOff

\newcommand{\refLine}[1]{line~\ref{#1}}
\newcommand{\refLines}[2]{lines~\ref{#1}--\ref{#2}}

\definealphabet{bb}{\mathbb}{A}{Z}
\definealphabet{rm}{\mathrm}{A}{Z}
\definealphabet{cal}{\mathcal}{A}{Z}
\definealphabet{frak}{\mathfrak}{a}{z}

\newtheorem{theorem}{Theorem}
\newtheorem{proposition}{Proposition}
\newtheorem{lemma}{Lemma}

\newtheorem{assumption}{Assumption}
\newtheorem*{problemsetting}{Problem Setting}

\theoremstyle{definition}

\theoremstyle{remark}

\crefname{theorem}{Theorem}{Theorems}
\crefname{proposition}{Proposition}{Propositions}
\crefname{lemma}{Lemma}{Lemmas}
\crefname{assumption}{Assumption}{Assumptions}
\crefname{definition}{Definition}{Definitions}
\crefname{corollary}{Corollary}{Corollaries}
\crefname{remark}{Remark}{Remarks}
\crefname{example}{Example}{Examples}
\crefname{figure}{Figure}{Figures}
\crefname{table}{Table}{Tables}
\crefname{algorithm}{Algorithm}{Algorithms}

\allowdisplaybreaks[4]

\newcommand{\feps}{\epsilon_\mathrm{f}}
\newcommand{\gtol}{\epsilon_{\mathrm{gtol}}}

\begin{document}

\title{Practical Regularized Quasi-Newton Methods\\
    with Inexact Function Values}

\author[1]{Hiroki Hamaguchi\,\orcidlink{0009-0005-7348-1356}\footnote{E-mail: \href{mailto:hirokihamaguchi0331@gmail.com}{\nolinkurl{hirokihamaguchi0331@gmail.com}}}}
\author[1]{Naoki Marumo\,\orcidlink{0000-0002-7372-4275}}
\author[1,2]{Akiko Takeda\,\orcidlink{0000-0002-8846-4496}}

\affil[1]{Graduate School of Information Science and Technology, The University of Tokyo, Tokyo, Japan}
\affil[2]{Center for Advanced Intelligence Project, RIKEN, Tokyo, Japan}

\maketitle

\begin{abstract}
    Many practical optimization problems involve objective function values that are corrupted by unavoidable numerical errors.
In smooth nonconvex optimization, quasi-Newton methods combined with line search are widely used due to their efficiency and scalability.
These methods implicitly assume accurate function evaluations and thus may fail to converge in noisy settings.
Developing fast and stable quasi-Newton methods for such scenarios is therefore crucial.

To address this issue, we propose a novel noise-tolerant regularized quasi-Newton method equipped with a relaxed Armijo-type line search, designed to remain stable under inaccurate function evaluations.
By combining a regularization parameter update rule inspired by Objective-Function-Free Optimization and the AdaGrad-Norm method, \MPCchanged{we prove an order-optimal $\mathcal{O}(\varepsilon^{-2})$ complexity bound under errors in function values.}

We performed extensive experiments on the CUTEst benchmark collection with artificially noisy objective function evaluations, as well as with low-precision floating-point arithmetic (64-, 32-, and 16-bit).
The results demonstrate \MPCchanged{stable performance under reduced precision and function-value noise} while retaining practical convergence speed and computational cost.

\end{abstract}

\section{Introduction}\label{sec:introduction}

We consider the unconstrained optimization problem
\begin{equation}
    \underset{x \in \bbR^{n}}{\mathrm{minimize}} \quad f(x),
    \label{prob:main_problem}
\end{equation}
where $f\colon \mathbb{R}^n \to \mathbb{R}$ is a continuously differentiable nonconvex function, and $x \in \mathbb{R}^n$ is the optimization variable.
We focus on settings where objective function evaluations, especially the function values, are contaminated by bounded and non-diminishing numerical noise~\citep{sunTrustRegionMethod2023}.
Specifically, only inexact function values $\overline{f}(x)$ are available for optimization.
Such inaccuracies are unavoidable in many realistic scenarios, including finite-precision arithmetic, simulation-based evaluations, and stochastic approximations~\citep{shiNoiseTolerantQuasiNewtonAlgorithm2022,moreEstimatingComputationalNoise2011,caflischMonteCarloQuasiMonte1998}.

In the absence of noise, \MPCchanged{quasi-Newton and regularized quasi-Newton methods can solve the large-scale problem~\eqref{prob:main_problem}, and their theoretical guarantees and practical efficiency have been well established.
For their success, the accuracy of function values is crucial.
Accurate function values are used to determine whether a step is accepted or rejected, and to adjust the step size or regularization parameter~\citep{2020SciPy-NMeth,Okazaki2007libLBFGS,uedaConvergencePropertiesRegularized2010,uedaRegularizedNewtonMethod2014,kanzowRegularizationLimitedMemory2023}.
When function values are contaminated by numerical noise, these mechanisms become unreliable, and the algorithms may lose their theoretical guarantees, behave erratically, or terminate prematurely with unreliable solutions.}

In the presence of noise, \MPCchanged{there are mainly two approaches to address function-value errors.
The first is to use noise-tolerant methods that account for bounded errors, such as the noise-tolerant quasi-Newton (NTQN) method~\citep{shiNoiseTolerantQuasiNewtonAlgorithm2022}.
The second is to avoid function values entirely, as in Objective-Function-Free Optimization (OFFO)~\citep{Gratton08022024,grattonOFFOMinimizationAlgorithms2023}.
However, the NTQN convergence analysis covers only strongly convex objectives, and totally discarding function values can limit practical efficiency when reliable function values are available.}

These observations naturally motivate us to design noise-tolerant quasi-Newton methods that use function values when they are reliable and use noise-aware strategies when they are not.
\MPCchanged{A natural idea is to switch between quasi-Newton steps guided by function values and OFFO-inspired regularized steps.
Still, overlooking function-value errors can compromise convergence guarantees, whereas overly conservative switching can slow practical convergence.
The key difficulty is therefore to determine when to switch while maintaining both convergence guarantees and practical efficiency.}

In this work, we propose a noise-tolerant regularized quasi-Newton method.
Our approach combines Armijo-type line search with quadratic regularization~\citep{uedaRegularizedNewtonMethod2014,kanzowRegularizationLimitedMemory2023} and an OFFO-inspired strategy~\citep{Gratton08022024,grattonOFFOMinimizationAlgorithms2023}.
When a sufficient decrease in function value is observed, the method behaves similarly to standard quasi-Newton methods with line search, preserving its practical efficiency.
Otherwise, an OFFO-inspired cumulative-gradient scale controls the regularization and supports convergence under function-value errors.

\MPCchanged{The key methodological advance is the carefully designed, error-model-based switch between unregularized quasi-Newton steps and OFFO-inspired regularized steps.
A relaxed Armijo condition alone may accept steps without sufficient decrease in the true objective, so its decrease estimates need not establish convergence to arbitrarily small first-order tolerances.
Our switch permits unregularized quasi-Newton steps only when the observed progress is sufficient to account for function-value errors.
This allows the method to exploit reliable function values while retaining convergence guarantees without applying OFFO-inspired regularization throughout.
The resulting method achieves an $\order{\varepsilon^{-2}}$ iteration bound, preserving the optimal worst-case order for smooth nonconvex optimization under the problem setting of persistent function-value errors and exact gradients.}

We also conducted extensive numerical experiments on the CUTEst benchmark collection to compare the analyzed method and a heuristic variant with existing methods.
We evaluated their performance under function-value noise within the theoretical problem setting, and under joint function-and-gradient noise and reduced-precision evaluations as complementary empirical tests beyond this setting.
We also assessed sensitivity to error-bound misspecification.
Both proposed variants attain higher solved percentages than the existing quasi-Newton baselines under artificial noise and at reduced precision.
A fixed-iteration timing experiment shows lower overhead than the regularized L-BFGS baselines, and a separate time-to-solution comparison shows favorable elapsed-time performance under reduced precision.

\subsection{Comparison With Related Work}
\label{sec:related_work}

\begin{table*}[t]
    \begin{MPCchange}
    \ifdefined\MPCRevision\captionsetup{font={color=red}}\fi
    \caption{
        Theoretical comparison of the proposed method, related Newton-type, quasi-Newton, and OFFO methods.
        The entries summarize theoretical guarantees, rather than the settings considered in numerical experiments.
        A checkmark (\ding{51}) denotes that the corresponding property holds or is covered by the cited theory, and ``--'' denotes that it does not hold or is not established by that theory.
        The ``Noisy $f$,'' ``Noisy $\nabla f$,'' and ``Nonconvex'' columns indicate whether the corresponding convergence theory allows function-value errors, allows gradient errors, and covers nonconvex objectives, respectively.
        For OFFO, function-value errors have no effect because function values are not evaluated.
    }
    \label{tab:related_work_comparison}
    \begin{centering}
    \begin{tabular}{@{}>{\raggedright\arraybackslash}p{0.27\linewidth}ccccc@{}}
            \toprule
            Method
             & Uses $f$
             & Noisy $f$
             & Noisy $\nabla f$
             & Nonconvex
             & $\order{\varepsilon^{-2}}$ guarantee \\
            \midrule
            \makecell[l]{L-BFGS~\citep{liuLimitedMemoryBFGS1989,Okazaki2007libLBFGS,2020SciPy-NMeth}}
             & \ding{51}
             & --
             & --
             & \ding{51}
             & -- \\
            \makecell[l]{Regularized Newton~\citep{uedaConvergencePropertiesRegularized2010}}
             & \ding{51}
             & --
             & --
             & \ding{51}
             & \ding{51} \\
            \makecell[l]{Regularized L-BFGS~\citep{kanzowRegularizationLimitedMemory2023}}
             & \ding{51}
             & --
             & --
             & \ding{51}
             & -- \\
            \makecell[l]{NTQN~\citep{xieAnalysisBFGSMethod2020,shiNoiseTolerantQuasiNewtonAlgorithm2022}}
             & \ding{51}
             & \ding{51}
             & \ding{51}
             & --
             & -- \\
            \makecell[l]{OFFO~\citep{Gratton08022024}}
             & --
             & \ding{51}
             & --
             & \ding{51}
             & \ding{51} \\
            \makecell[l]{Proposed~(\cref{alg:regularized_qn})}
             & \ding{51}
             & \ding{51}
             & --
             & \ding{51}
             & \ding{51} \\
            \bottomrule
        \end{tabular}
    \par
    \end{centering}
    \end{MPCchange}
\end{table*}

\begin{MPCchange}
We next discuss the methods in \cref{tab:related_work_comparison} in order, focusing on their assumptions and convergence guarantees.

The first three methods in the table are standard quasi-Newton and regularized Newton-type methods.
The Broyden--Fletcher--Goldfarb--Shanno (BFGS) method and its Limited-memory variant (L-BFGS) are among the most efficient and widely used algorithms to solve the problem~\eqref{prob:main_problem}~\citep{liuLimitedMemoryBFGS1989,nocedal1999numerical}.
They generally employ line search procedures that enforce the Wolfe conditions to guarantee sufficient descent and the positive definiteness of the Hessian approximation.
Regularized Newton methods attain an $\order{\varepsilon^{-2}}$ iteration bound under their stated assumptions and with accurate function values~\citep[Theorem~4.2]{uedaConvergencePropertiesRegularized2010}.
The cited analysis of regularized L-BFGS establishes only asymptotic global convergence to first-order stationarity~\citep[Theorem~2]{kanzowRegularizationLimitedMemory2023}.
As discussed above, their theoretical guarantees rely on accurate function values, and do not cover function-value errors.
Their practical performance may degrade in noisy environments.

NTQN addresses bounded errors in function and gradient evaluations through a noise-tolerant line search and a lengthening procedure~\citep{shiNoiseTolerantQuasiNewtonAlgorithm2022}.
The NTQN analysis allows bounded gradient errors and establishes convergence to a noise-dependent neighborhood for strongly convex objectives~\citep[Theorem~3.5]{shiNoiseTolerantQuasiNewtonAlgorithm2022}.
Its practical line search adds twice the absolute function-error bound to the Armijo right-hand side after an unsuccessful initial trial.
Our relaxation instead follows a hybrid absolute-relative error model and is recomputed at each trial point.

OFFO methods obtain convergence guarantees without evaluating function values~\citep{Gratton08022024,grattonOFFOMinimizationAlgorithms2023}.
The Adagrad-like class analyzed by Gratton et al.~\citep{Gratton08022024} attains an $\order{\varepsilon^{-2}}$ iteration bound for smooth nonconvex objectives, but does not use function-value information to guide the steps.

Lastly, our method uses inexact function values to permit unregularized quasi-Newton steps when progress is certified and retains this worst-case order under persistent function-value errors.
\end{MPCchange}

\subsection{Organization of the Paper}

The remainder of this paper is organized as follows.
In \cref{sec:preliminaries}, we review relevant background and introduce the problem setting.
In \cref{sec:algorithm}, we present the proposed algorithm.
In \cref{sec:analysis}, we establish a global convergence property of the proposed algorithm.
In \cref{sec:proposed_choice}, we discuss practical choices of algorithmic variables and parameters.
In \cref{sec:experiments}, we report numerical results on the CUTEst benchmark collection.
Finally, in \cref{sec:conclusion}, we conclude the paper with a summary and future research directions.

Our implementation is publicly available on GitHub%
\footnote{\label{fn:ourGitHub}\url{https://github.com/HirokiHamaguchi/qnlab}}.

\providecommand{\resultsNoiseMainWidth}{0.52}
\providecommand{\resultsNoiseLegendWidth}{0.85}
\providecommand{\resultsAllPrecisionsSubfigWidth}{0.25}

\providecommand{\resultsNoiseMainWidth}{0.68}
\providecommand{\resultsNoiseLegendWidth}{0.98}
\providecommand{\resultsAllPrecisionsSubfigWidth}{0.32}
\providecommand{\timeToSolutionFigure}{\begin{figure}[htbp]}

\section{Preliminaries and Problem Setting}
\label{sec:preliminaries}

We begin by reviewing relevant background and formalizing the problem setting.

\subsection{Regularized Quasi-Newton Method}
\label{sec:prel_rqn}

Regularized quasi-Newton methods control the search direction by adding a regularization term to the Hessian approximation.
There has been renewed interest in regularized Newton-type methods, including cubic-regularization schemes~\citep{nesterovCubicRegularizationNewton2006,ranganathSymmetricRankOneQuasiNewton2025,wangConvergenceRatesRegularized2025}, and we focus on quadratic regularization~\citep{kanzowRegularizationLimitedMemory2023,doikovSuperUniversalRegularizedNewton2024,liTruncatedRegularizedNewton2009,zhangNewRegularizedQuasiNewton2018,uedaConvergencePropertiesRegularized2010,uedaRegularizedNewtonMethod2014,shengNewAdaptiveTrust2018} for simplicity of analysis and implementation.

For $g_k \defeq \nabla f(x_k)$ and an approximate Hessian $B_k \in \bbR^{n \times n}$, regularized quasi-Newton methods compute the step as
\begin{equation*}
    x_{k+1} = x_k + d_k, \quad d_k = -(B_k + \mu_k I)^{-1} g_k,
\end{equation*}
with a regularization parameter $\mu_k \ge 0$.
Combining regularization with line search is also possible and has been explored~\citep{uedaConvergencePropertiesRegularized2010}.
Several strategies for choosing $\mu_k$ in nonconvex optimization have been proposed~\citep{uedaConvergencePropertiesRegularized2010,uedaRegularizedNewtonMethod2014,shengNewAdaptiveTrust2018}, and one of them is given by
\begin{equation*}
    \mu_k = c_1 \max(0, -\lambda_{\min}(B_k)) + c_2 \norm{\nabla f(x_k)}^\delta,
\end{equation*}
where $c_1 > 1$, $c_2 > 0$, $\delta \ge 0$, and $\lambda_{\min}(B_k)$ denotes the smallest eigenvalue of $B_k$.
For symmetric $B_k$ and $\nabla f(x_k) \neq 0$, this choice makes $B_k + \mu_k I$ positive definite and yields a descent direction.
We adapt such regularization strategies to noisy evaluation environments.

\subsection{Objective-Function-Free Optimization}
\label{sec:prel_offo}

Objective-Function-Free Optimization (OFFO) refers to methods that do not evaluate the objective function~\citep{Gratton08022024,grattonOFFOMinimizationAlgorithms2023}.
Such methods are particularly attractive when function values are unreliable due to noise.
Similar adaptive trust-region approaches are also discussed~\citep{grapigliaAdaptiveTrustregionMethod2022}.

Prominent examples of OFFO include adaptive gradient methods such as AdaGrad~\citep{duchiAdaptiveSubgradientMethods2011} and AdaGrad-Norm~\citep{wardAdaGradStepsizesSharp2020}.
Their update rules are
\begin{align*}
    \text{(AdaGrad)} \quad x_{k+1}^{(i)}                                      & = x_k^{(i)} -
    \frac{1}{\theta \sqrt{\varsigma + \sum_{j=0}^k (g_j^{(i)})^2}} g_k^{(i)}, &               & (i=1,2,\ldots,n) \\
    \text{(AdaGrad-Norm)} \quad x_{k+1}                                       & = x_k -
    \frac{1}{\theta \sqrt{\varsigma + \sum_{j=0}^k \norm{g_j}^2}} g_k
\end{align*}
respectively, where $\varsigma>0$ and $\theta>0$ are algorithmic parameters, and $x^{(i)}$ denotes the $i$-th component of a vector $x$.

\subsection{Problem Setting}
\label{sec:problem_setting}

We consider optimization problems in which exact objective function values are unavailable, following the recent literature on error-aware optimization~\citep{grattonNoteSolvingNonlinear2020,bellaviaImpactNoiseEvaluation2021,carterNumericalExperienceClass1993,clancyTROPHYTrustRegion2022,sunTrustRegionMethod2023,ruanAdaptiveRegularizedQuasiNewton2024}.
When both objective values and gradients are subject to uncontrolled errors, reliable optimization cannot be guaranteed without additional assumptions.
Thus, meaningful optimization requires that at least one source of information be sufficiently accurate.
This work focuses on problems arising from finite-precision computations or inexact evaluations, where objective function evaluations are subject to non-negligible errors while gradient information can be computed with relatively high accuracy.
Settings dominated by gradient errors, such as stochastic approximation and inexact-gradient methods~\citep{berahasTheoreticalEmpiricalComparison2021,xieAnalysisBFGSMethod2020,shiNoiseTolerantQuasiNewtonAlgorithm2022}, are beyond the scope of this work.

We formalize this problem setting through the following conditions on function values and gradients.
\begin{problemsetting}
\MPCchanged{For a fixed error rate $0 \le \feps < 1$, the available function value $\overline{f}(x)$ satisfies
\begin{equation}\label{eq:error_model}
    \abs{f(x)-\overline{f}(x)} \le \feps \max(1,\abs{f(x)}).
\end{equation}
The exact gradient $\nabla f(x)$ is available at every evaluated point.}
\end{problemsetting}

We adopt the exact-gradient condition in the problem setting as a modeling simplification that keeps the analysis focused on function-value errors, rather than as a claim that practical gradient evaluations are always exact.
The numerical experiments also examine stability with gradient errors beyond this setting (see \cref{sec:experiments}).

The hybrid absolute-relative error model captures the typical behavior of rounding errors in IEEE~754 floating-point arithmetic~\citep{highamAccuracyStabilityNumerical2002,10.5555/3168451.3168462}.
The error is predominantly relative when $\abs{f(x)}$ is large and effectively absolute when $\abs{f(x)}$ is small.
\MPCchanged{In the numerical experiments, $\feps$ is supplied as a safety margin for accumulated errors in the computed function values, with its practical specification and sensitivity to misspecification discussed in \cref{sec:experiments}.}
Closely related models are also employed in practical stopping criteria~\citep{2020SciPy-NMeth}.

\section{Proposed Algorithm}\label{sec:algorithm}

We propose a new algorithm for the problem setting stated in \cref{sec:problem_setting}.
We draw inspiration from the OFFO framework to design regularization and scaling strategies for inexact function values.
\MPCchanged{Unlike purely OFFO methods, our algorithm uses function values in the line search in both modes, while the OFFO-inspired regularization scale depends only on gradient information.}
The pseudo-code of our algorithm is given in \cref{alg:regularized_qn}.
The minimum of the empty set is defined as $+\infty$ in \refLine{line:mu_0_condition}.
\Cref{alg:backtracking_linesearch} is an auxiliary line search algorithm called in \cref{alg:regularized_qn}.
In the following, we explain the three main components of \cref{alg:regularized_qn}.

\DontPrintSemicolon
\begin{algorithm}[ht]
    \normalsize
    \KwIn{$x_0$, $0 < k_{\max}$, $0 \le \feps < 1$, $0 < \theta_{\min} \leq \theta_{\max}$ and $0 < \varsigma < 1$}
    $k \gets 0$, $g_0 \gets \nabla f(x_0)$, \MPCchanged{$K^0_0 \gets \emptyset$, $K^+_0 \gets \emptyset$} \label{line:init}\;
    \For{$k=0,1,2,\ldots, k_{\max}-1$}{
        \uIf{\MPCchanged{$\min_{j \in K^0_k} (\overline{f}(x_j)- \Delta_j) \ge \overline{f}(x_k)$}\label{line:mu_0_condition}}{
            \MPCchanged{$K^0_{k+1} \gets K^0_k \cup \{k\}$, $K^+_{k+1} \gets K^+_k$} \label{line:update_K_0}\;
            $\mu_k \gets 0$ \label{line:update_mu_0}\;
        }
        \uElse{
            \MPCchanged{$K^0_{k+1} \gets K^0_k$, $K^+_{k+1} \gets K^+_k \cup \{k\}$} \label{line:update_K_plus}\;
            \MPCchanged{$\mu_{k} \gets \theta_k \sqrt{\varsigma + \sum_{j \in K^+_{k+1}} \norm{g_{j}}^2}$} with $\theta_k \in [\theta_{\min}, \theta_{\max}]$ ~ (see \cref{sec:proposed_choice_mu}). \label{line:update_theta}\;
        }
        $d_k \gets -(B_k + \mu_k I)^{-1} g_k$ with $B_k$ ~ (see \cref{sec:proposed_choice_Bk}). \label{line:d_k}\;
        Find $\alpha_k$ and $\Delta_k$ satisfying \cref{eq:relaxed_Armijo_condition} by \cref{alg:backtracking_linesearch}.  \label{line:armijo}\;
        $x_{k+1} \gets x_k + \alpha_k d_k$,\,
        $g_{k+1} \gets \nabla f(x_{k+1})$ \label{line:update}\;
    }
    \caption{Noise-Tolerant Regularized Quasi-Newton Method}
    \label{alg:regularized_qn}
\end{algorithm}

\DontPrintSemicolon
\begin{algorithm}[ht]
    \normalsize
    \KwIn{$x_k \in \bbR^n$, $d_k \in \bbR^n$, $g_k \in \bbR^n$, $0< c < 1$ and $0<\beta_{\min} < \beta_{\max} < 1$}
    $\alpha_k \gets 1$\; \label{line:ls_init}
    $\Delta_k \gets \frac{2\feps}{1-\feps} \max(1, \overline{f}(x_k), -\overline{f}(x_k+\alpha_k d_k))$\;
    \While{$\overline{f}(x_k) + c \alpha_k g_k^\top d_k + \Delta_k < \overline{f}(x_k+\alpha_k d_k)$}{
        $\alpha_k \leftarrow \beta \alpha_k$ with $\beta \in [\beta_{\min},\beta_{\max}]$ ~ (see \cref{sec:proposed_choice_alpha}). \; \label{line:ls_backtracking}
        $\Delta_k \gets \frac{2\feps}{1-\feps} \max(1, \overline{f}(x_k), -\overline{f}(x_k+\alpha_k d_k))$\;
    }
    \Return{$\alpha_k, \Delta_k$}\;
    \caption{Backtracking Line Search with Relaxed Armijo Condition}
    \label{alg:backtracking_linesearch}
\end{algorithm}

\paragraph{Computation of the Regularized Quasi-Newton Direction (\refLine{line:d_k}): }\label{sec:algorithm_components_d}

The first component is the computation of the regularized quasi-Newton direction $d_k$ stated in \cref{sec:prel_rqn}.
Given $g_k=\nabla f(x_k)$, an approximate Hessian $B_k$, and a regularization parameter $\mu_k$, we compute the search direction $d_k = -(B_k+\mu_k I)^{-1} g_k$.
A practical choice of $B_k$ is discussed in \cref{sec:proposed_choice_Bk}.

\paragraph{Relaxed Armijo Condition with an Error-Absorbing Term (\refLine{line:armijo}): }\label{sec:algorithm_components_alpha}

The second component is the computation of the step size $\alpha_k$ using a relaxed Armijo condition, which is defined as follows:
\begin{equation}\label{eq:relaxed_Armijo_condition}
    \begin{dcases}
        \overline{f}(x_k)+c\alpha_k g_k^\top d_k+\Delta_k\ge\overline{f}(x_k+\alpha_k d_k), \\
        \Delta_k = \frac{2\feps}{1-\feps} \max(1, \overline{f}(x_k), -\overline{f}(x_k + \alpha_k d_k)),
    \end{dcases}
\end{equation}
where $\Delta_k$ is the error-absorbing term that is recomputed for each $\alpha_k$.
\MPCchanged{When the maximum in the definition of $\Delta_k$ equals $1$, the relaxation reduces to $2\feps/(1-\feps)$, giving a constant allowance analogous to the absolute-error relaxation used by Shi et al.~\citep[Section~4.2]{shiNoiseTolerantQuasiNewtonAlgorithm2022}.}
\MPCchanged{As shown later in the proof of \cref{lem:error_bound_on_function_evaluations}, $\Delta_k$ bounds the absolute discrepancy between the true and observed function decreases whenever the true function value decreases.}
This relaxation guarantees the existence of $\alpha_k \in (0,1]$ even in the presence of errors in $\overline{f}$.
When $d_k$ is a descent direction ($g_k^\top d_k < 0$), \cref{eq:relaxed_Armijo_condition} ensures that a temporary increase of $\overline{f}$ is at most $\Delta_k$.
An increase in the trial function value does not enlarge the relaxation, since $\Delta_k$ depends on $-\overline{f}(x_k + \alpha_k d_k)$ rather than $\overline{f}(x_k + \alpha_k d_k)$.

\paragraph{Adaptive Update of the Regularization Parameter (\refLines{line:mu_0_condition}{line:update_theta}): }\label{sec:algorithm_components_mu}

The third component is the adaptive update of the regularization parameter $\mu_k$.
A larger $\mu_k$ yields a more conservative step, while a smaller $\mu_k$ allows a more aggressive quasi-Newton step, especially when $\mu_k=0$.
In practical settings, $\mu_k=0$ is preferred, since this allows us to fully exploit the quasi-Newton approximation and typically accelerates practical convergence.
\begin{MPCchange}
Let $K\coloneqq \{0,1,2,\ldots, k_{\max}-1\}$ be the set of iteration indices.
As shown in \cref{alg:regularized_qn}, $K_k^0 = \{j \in K \mid \mu_j = 0,\ j < k\}$ and $K_k^+ = \{j \in K \mid \mu_j > 0,\ j < k\}$ contain the unregularized and regularized iteration indices classified before iteration $k$, respectively.
After the final iteration, define
\begin{equation}\label{eq:def_K0_Kplus}
    K^0 \coloneqq K^0_{k_{\max}}, \quad K^+ \coloneqq K^+_{k_{\max}}.
\end{equation}
These final sets partition $K$.
\end{MPCchange}
We set $\mu_k=0$ if and only if the following condition is satisfied:
\begin{equation}\label{eq:mu_0_condition}
    \MPCchanged{\min_{j \in K_k^0} \qty(\overline{f}(x_j)- \Delta_j) \ge \overline{f}(x_k).}
\end{equation}
Note that $k=0$ satisfies \cref{eq:mu_0_condition} since $K_0^0$ is empty.
This condition restricts $\mu_k=0$ to iterations for which a sufficient decrease in $\overline f$ has been observed.
Thus, although $\overline{f}$ can temporarily increase, every iteration in $K^0$ must improve the previous error-adjusted reference value.
Otherwise, to ensure convergence to first-order stationary points, we apply an OFFO-based update
\begin{equation}\label{eq:offo_based_update}
    \mu_k = \theta_k \sqrt{\varsigma+\sum_{j \in K^+_{k+1}}\norm{g_j}^2},
\end{equation}
as described in \cref{sec:prel_offo}.
Since OFFO schemes do not rely on function values, this update is independent of numerical errors in $\overline{f}$.

\section{Theoretical Analysis}\label{sec:analysis}

In this section, we establish the global convergence properties of \cref{alg:regularized_qn} under the problem setting in \cref{sec:problem_setting}.

\subsection{Assumptions}\label{sec:assumptions}

Throughout this section, we make the following three assumptions.
The first assumption ensures that the objective function is bounded below.
\begin{assumption}\label{ass:bounded}
    The function $f$ is bounded below, meaning that for all $x \in \mathbb{R}^n$,
    \begin{equation*}
        f(x) \ge f^* \defeq \inf_{x \in \bbR^n} f(x) > -\infty .
    \end{equation*}
\end{assumption}
By the triangle inequality and \cref{eq:error_model}, we have
$\overline{f}(x) \ge f(x) - \feps \max(1, \abs{f(x)})$.
Since $t - \feps \max(1, \abs{t})$ is a continuous, non-decreasing function of $t$ for $0 \le \feps < 1$, we obtain
\begin{equation*}
    \overline{f}(x) \ge
    \overline{f}^*
    \defeq \inf_{x \in \bbR^n} \qty( f(x) - \feps \max(1, \abs{f(x)}) )
    = f^* - \feps \max(1, \abs{f^*})
    > -\infty .
\end{equation*}

The second assumption ensures smoothness of the objective function.
\begin{assumption}\label{ass:LSmooth}
    The function $f \colon \bbR^n \to \bbR$ is $L$-smooth, meaning that there exists a constant $L > 0$ such that for all $x, y \in \bbR^n$,
    \begin{equation*}
        f(y) \le f(x) + \nabla f(x)^\top (y - x) + \frac{L}{2} \norm{y - x}^2.
    \end{equation*}
\end{assumption}

The third assumption requires the matrices $\{B_k\}_{k=0}^{k_{\mathrm{max}}-1}$ to be uniformly positive definite and bounded.
\begin{assumption}\label{ass:BBounds}
    There exist constants $0 < m \leq M$ such that all $B_k$ satisfy
    \begin{equation*}
        mI \preceq B_k \preceq MI.
    \end{equation*}
\end{assumption}
By \cref{ass:BBounds}, we can derive useful bounds involving $g_k$ and $d_k$.
Using $d_k=-(B_k+\mu_k I)^{-1} g_k$ (\refLine{line:d_k}) and \cref{ass:BBounds} asserting $(m + \mu_k)I \preceq B_k + \mu_k I \preceq (M + \mu_k)I$, we obtain the following bounds:
\begin{align}
    -g_k^\top d_k & = d_k^\top (B_k+\mu_k I) d_k      \ge (m+\mu_k)\norm{d_k}^2 \ge m \norm{d_k}^2, \label{eq:BBounds_results_zero} \\
    -g_k^\top d_k & = g_k^\top (B_k+\mu_k I)^{-1} g_k \ge \frac{1}{M+\mu_k}\norm{g_k}^2, \label{eq:BBounds_results_first}           \\
    \norm{d_k}^2  & = \norm{(B_k+\mu_k I)^{-1}g_k}^2 \le \frac{1}{(m+\mu_k)^2}\norm{g_k}^2. \label{eq:BBounds_results_second}
\end{align}

\subsection{Error Bound on Function Evaluations}\label{sec:error_bound_on_function_evaluations}

We first establish useful properties of the error-absorbing term $\Delta_k$ defined in \cref{eq:relaxed_Armijo_condition}.
In the following, we define the actual and observed function decreases as
\begin{align}
    \delta_k            & \defeq f(x_k) - f(x_{k+1}),           \label{eq:def_small_delta_k}                      \\
    \overline{\delta}_k & \defeq \overline{f}(x_k) - \overline{f}(x_{k+1}). \label{eq:def_overline_small_delta_k}
\end{align}
The first property states that the error between the actual decrease $\delta_k$ and the observed decrease $\overline{\delta}_k$ can be bounded by $\Delta_k$ when $f(x_{k+1}) \le f(x_k)$.
\begin{lemma}\label{lem:error_bound_on_function_evaluations}
    If $x_k$ and $x_{k+1}$ satisfy $f(x_{k+1}) \le f(x_k)$, it holds that
    \begin{equation*}
        \delta_k \leq \overline{\delta}_k + \Delta_k.
    \end{equation*}
\end{lemma}
\begin{proof}
    We start from general observations.
    We first prove that, for all $x \in \bbR^n$,
    \begin{equation}\label{eq:fx_two_cases}
        \begin{cases}
            \max(1,+f(x))  \le \frac{1}{1-\feps} \max(1,+\overline f(x)), \\
            \max(1,-f(x)) \le \frac{1}{1-\feps} \max(1,-\overline f(x)).
        \end{cases}
    \end{equation}
    We prove the inequality with the plus sign.
    If $+f(x)\le 1$, then $\max(1,+f(x)) = 1 \leq \frac{1}{1-\feps} \leq \frac{1}{1-\feps} \max(1,+\overline f(x))$.
    If $+f(x) > 1$, \cref{eq:error_model} implies
    \begin{equation*}
        \abs{f(x)-\overline{f}(x)} \le \feps \max(1,\abs{f(x)}) = +\feps f(x)
    \end{equation*}
    and thus $f(x) - \overline f(x) \leq +\feps f(x)$. Hence,
    \begin{equation*}
        \max(1,+f(x)) = +f(x) \le +\frac{1}{1-\feps} \overline f(x) \le \frac{1}{1-\feps} \max(1,+\overline f(x)).
    \end{equation*}
    The case with the negative sign follows analogously, using \cref{eq:error_model} to obtain $-f(x) + \overline f(x) \leq -\feps f(x)$ when $f(x)<-1$.
    This completes the proof of \cref{eq:fx_two_cases}.

    We also have the following general bound for all $k \in K$:
    \begin{align}
        \abs{\delta_k-\overline{\delta}_k}
         & = \abs{(f(x_k)-\overline{f}(x_k)) - (f(x_{k+1})-\overline{f}(x_{k+1}))}      &  & (\text{rearrangement})                                \notag \\
         & \leq \abs{f(x_k)-\overline{f}(x_k)} + \abs{f(x_{k+1})-\overline{f}(x_{k+1})} &  & (\text{triangle inequality}) \notag                          \\
         & \le \feps\max(1,\abs{f(x_k)}) + \feps\max(1,\abs{f(x_{k+1})})                &  & (\text{\cref{eq:error_model}}) \notag                        \\
         & \le 2\feps \max(1,\abs{f(x_k)},\abs{f(x_{k+1})}).                            &  & (\text{max property}) \label{eq:delta_k_error_bound}
    \end{align}

    With these preparations, we now prove the lemma.
    From the condition $f(x_{k+1}) \le f(x_k)$, \cref{eq:delta_k_error_bound} can be further bounded as
    \begin{align*}
        2 \feps \max(1,\abs{f(x_k)},\abs{f(x_{k+1})})
         & = 2 \feps \max(1,f(x_k),-f(x_{k+1}))                                        &  & (\text{$f(x_{k+1}) \le f(x_k)$}) \\
         & \le \frac{2 \feps}{1-\feps}\max(1,\overline{f}(x_k),-\overline{f}(x_{k+1})) &  & (\text{\cref{eq:fx_two_cases}})  \\
         & = \Delta_k,                                                                 &  &
        (\text{\cref{eq:relaxed_Armijo_condition}})
    \end{align*}
    which yields the desired inequality with $\delta_k - \overline{\delta}_k \leq \abs{\delta_k - \overline{\delta}_k} \leq \Delta_k$.
    \myQED
\end{proof}

We also derive the following lemma for later use.
\begin{lemma}\label{lem:error_delta_bounded_by_Delta}
    Under \cref{ass:BBounds}, it holds that
    \begin{equation*}
        \overline{\delta}_k \leq \delta_k + \frac{1 + \feps}{1 - \feps} \Delta_k.
    \end{equation*}
\end{lemma}
\begin{proof}
    Since $-g_k^\top d_k \geq 0$ from \cref{ass:BBounds} and \cref{eq:BBounds_results_zero}, the relaxed Armijo condition (\cref{eq:relaxed_Armijo_condition}) implies
    \begin{equation}\label{eq:overline_delta_k_lower_bound}
        \overline{f}(x_k) - \overline{f}(x_{k+1})
        \geq - c \alpha_k g_k^\top d_k -\Delta_k
        \geq -\Delta_k.
    \end{equation}
    Thus, using the general results in \cref{eq:delta_k_error_bound} and $\Delta_k \geq 0$, we can further bound as
    \begin{align*}
        \abs{\delta_k - \overline{\delta}_k} & \leq 2\feps \max(1,\abs{f(x_k)},\abs{f(x_{k+1})})                                                                                 &  & (\text{\cref{eq:delta_k_error_bound}})          \\
                                             & \leq \frac{2\feps}{1-\feps} \max(1, \overline{f}(x_{k}), -\overline{f}(x_{k}), \overline{f}(x_{{k}+1}), -\overline{f}(x_{{k}+1})) &  & (\text{\cref{eq:fx_two_cases}})                 \\
                                             & \le \frac{2\feps}{1-\feps} \max(1, \overline{f}(x_{k}) + \Delta_{k}, -\overline{f}(x_{{k}+1}) + \Delta_{k})                       &  & (\text{\cref{eq:overline_delta_k_lower_bound}}) \\
                                             & \le \frac{2\feps}{1-\feps} \qty(\max(1, \overline{f}(x_{k}), -\overline{f}(x_{{k}+1})) + \Delta_k)                                                                                     \\
                                             & = \qty(1 + \frac{2\feps}{1-\feps}) \Delta_k  = \frac{1 + \feps}{1 - \feps} \Delta_k,                                              &  & (\text{\cref{eq:relaxed_Armijo_condition}})
    \end{align*}
    which yields the desired inequality with $\overline{\delta}_k - \delta_k \leq \abs{\delta_k - \overline{\delta}_k} \leq \frac{1 + \feps}{1 - \feps} \Delta_k$.
    \myQED
\end{proof}

\subsection{Backtracking Stage}\label{sec:inner_loop}

We next analyze the line search in \cref{alg:backtracking_linesearch}.
Recall that the parameter $c$ in \cref{alg:backtracking_linesearch} satisfies $0 < c < 1$ and $0 < m$ from \cref{ass:BBounds}.

\begin{lemma}\label{lem:backtracking_finite}
    Under \cref{ass:LSmooth,ass:BBounds},
    the backtracking procedure in \cref{alg:backtracking_linesearch} terminates whenever the step size $\alpha_k$ satisfies
    \begin{equation}\label{eq:alpha_k_upper_bound}
        \alpha_k \le \frac{2(1-c)m}{L}.
    \end{equation}
\end{lemma}
\begin{proof}
    Under the $L$-smoothness of $f$ (\cref{ass:LSmooth}), we always have
    \begin{align}
        f(x_k)-f(x_k+\alpha_k d_k)
         & \ge -\alpha_k g_k^\top d_k - \frac{L}{2}\alpha_k^2\norm{d_k}^2                                   \notag                           \\
         & = -c \alpha_k g_k^\top d_k + \alpha_k \qty(-(1-c)g_k^\top d_k - \frac{L\alpha_k}{2}\norm{d_k}^2). \label{eq:backtracking_Lsmooth}
    \end{align}
    From \cref{ass:BBounds} and \cref{eq:BBounds_results_zero}, we have $-g_k^\top d_k \geq m \norm{d_k}^2$.
    Under \cref{eq:alpha_k_upper_bound}, we have
    \begin{equation}\label{eq:backtracking_simple_bound}
        -(1-c)g_k^\top d_k - \frac{L\alpha_k}{2}\norm{d_k}^2
        \ge \qty((1-c)m - \frac{L\alpha_k}{2})\norm{d_k}^2 \ge 0.
    \end{equation}
    Combining \cref{eq:backtracking_Lsmooth,eq:backtracking_simple_bound} yields
    \begin{equation*}
        f(x_k) - f(x_k+\alpha_k d_k) \geq -c \alpha_k g_k^\top d_k.
    \end{equation*}
    Since $-g_k^\top d_k \geq 0$ from \cref{ass:BBounds} and \cref{eq:BBounds_results_zero}, the condition of \cref{lem:error_bound_on_function_evaluations} is satisfied with $x_{k+1} = x_k + \alpha_k d_k$.
    Thus, we obtain $\delta_k \le \overline{\delta}_k + \Delta_k$, i.e.,
    \begin{equation*}
        -c\alpha_k g_k^\top d_k
        \le
        f(x_k) - f(x_k+\alpha_k d_k)
        \le \overline{f}(x_k) - \overline{f}(x_k+\alpha_k d_k) + \Delta_k.
    \end{equation*}
    This inequality means that the relaxed Armijo condition (\cref{eq:relaxed_Armijo_condition}) is satisfied and thus the backtracking procedure terminates.
    This completes the proof.
    \myQED
\end{proof}

Now, let us analyze \cref{alg:backtracking_linesearch} using \cref{lem:backtracking_finite}.
Define
\begin{equation}\label{eq:def_alpha_bar}
    \overline{\alpha}\defeq \beta_{\min}\min\qty(\frac{2(1-c)m}{L},1)>0.
\end{equation}
The following states the key property of the backtracking procedure.

\begin{proposition}\label{prop:backtracking_summary}
    Under \cref{ass:LSmooth,ass:BBounds}, the backtracking line search in \cref{alg:backtracking_linesearch} terminates after finitely many rejections, and the step size $\alpha_k$ satisfies
    \begin{equation*}
        \alpha_k\ge\overline{\alpha}.
    \end{equation*}
\end{proposition}
\begin{proof}
    Line~\ref{line:ls_backtracking} of \cref{alg:backtracking_linesearch} multiplies $\alpha_k$ by $\beta$ at each rejection, where $0<\beta_{\min} \leq \beta \leq \beta_{\max}<1$.
    For the number of backtracking rejections $r$, we have $\alpha_k\le\beta_{\max}^r$.
    Thus, $r$ is finite since the backtracking procedure terminates for sufficiently small $\alpha_k$ by \cref{lem:backtracking_finite}.
    If $r=0$, $\alpha_k=1\ge\overline{\alpha}$.
    Otherwise, $\alpha_k \geq \beta_{\min} \alpha'_k$, where $\alpha'_k$ is the last rejected step size.
    By \cref{lem:backtracking_finite}, we have $\alpha'_k > \frac{2(1-c)m}{L}$, and thus
    \begin{equation*}
        \alpha_k \ge \beta_{\min}\alpha'_k > \beta_{\min} \frac{2(1-c)m}{L} \geq \overline{\alpha}.
    \end{equation*}
    This completes the proof.
    \myQED
\end{proof}

\subsection{Bound for the case \texorpdfstring{$\mu_k=0$}{mu k = 0}}
\label{sec:unregularized_bound}

In this subsection, we lower bound the sum of actual function decreases $\delta_k$ over $k \in K^0$.
Recall that $K^0$ denotes the set of iteration indices $k$ such that $\mu_k = 0$, as in \cref{eq:def_K0_Kplus}.
To this end, we establish that the sum of $\Delta_k$ over $k \in K^0$ is upper bounded, where $\Delta_k$ is the error-absorbing term in the relaxed Armijo condition (\cref{eq:relaxed_Armijo_condition}).
We introduce the following constant $\Delta_{\max}$, which serves as an upper bound for $\Delta_k$ of $k \in K^0$.
\begin{equation}\label{eq:def_Delta_max}
    \Delta_{\max} \defeq \frac{2\feps}{1-\feps} \max(1, \overline{f}(x_{0}), -\overline{f}^*).
\end{equation}

\begin{lemma}\label{lem:delta_bound_and_sum}
    Under \cref{ass:bounded,ass:LSmooth,ass:BBounds}, the sum of $\Delta_k$ over $k \in K^0$ satisfies
    \begin{equation*}
        \sum_{k \in K^0} \Delta_k \le \overline{f}(x_0) - \overline{f}^* + \Delta_{\max}.
    \end{equation*}
\end{lemma}
\begin{proof}
    Let $(k_i)_{i=0}^{\ell}$ with $\ell \geq 0$ be the elements of $K^0$ in increasing order.
    We have $\overline{f}(x_{k_{\ell}}) \le \overline{f}(x_0)$ from \cref{eq:mu_0_condition} with $k=k_{\ell}$, where the case $\ell=0$ follows from $k_0=0$.
    Combining this with \cref{ass:bounded}, we have $\Delta_{k_{\ell}} \le \Delta_{\max}$.
    For all $i < \ell$, \cref{eq:mu_0_condition} with $j=k_i$ and $k=k_{i+1}$ yields
    \begin{equation}\label{eq:Delta_i_for_tele_sum}
        \Delta_{k_i} \leq \overline{f}(x_{k_i}) - \overline{f}(x_{k_{i+1}}),
    \end{equation}
    and summing $\Delta_{k}$ over $k \in K^0$ gives
    \begin{align*}
        \sum_{k \in K^0} \Delta_{k}
         & \le \sum_{i=0}^{\ell-1} (\overline{f}(x_{k_i}) - \overline{f}(x_{k_{i+1}})) + \Delta_{k_{\ell}} &  & (\text{\cref{eq:Delta_i_for_tele_sum}})                                                            \\
         & = \overline{f}(x_0) - \overline{f}(x_{k_{\ell}}) + \Delta_{k_{\ell}}                            &  & (\text{telescoping sum})                                                                           \\
         & \leq \overline{f}(x_0) - \overline{f}^* + \Delta_{\max}.                                        &  & (-\overline{f}(x_{k_{\ell}}) \le -\overline{f}^* \text{ and } \Delta_{k_{\ell}} \le \Delta_{\max})
    \end{align*}
    This completes the proof.
    \myQED
\end{proof}

We now bound the sum of $\delta_k = f(x_k) - f(x_{k+1})$ over $k \in K^0$ with $\norm{g_{k}}^2$ as follows.
\begin{lemma}\label{lem:sum_small_delta_k_bound_K0}
    Under \cref{ass:bounded,ass:LSmooth,ass:BBounds}, the sum of $\delta_k$ over $k \in K^0$ satisfies
    \begin{equation*}
        \sum_{k \in K^0} \delta_{k} \geq \frac{c \overline{\alpha}}{M} \sum_{k \in K^0} \norm{g_{k}}^2 - \frac{2}{1-\feps} \qty(\overline{f}(x_0) - \overline{f}^* + \Delta_{\max}).
    \end{equation*}
\end{lemma}
\begin{proof}
    The relaxed Armijo condition (\cref{eq:relaxed_Armijo_condition}) at iteration $k \in K^0$ yields
    \begin{align}
        \overline{\delta}_{k} + \Delta_{k}
         & = \overline{f}(x_{k}) - \overline{f}(x_{k+1}) + \Delta_{k} &  & (\text{\cref{eq:def_overline_small_delta_k}})                                   \notag     \\
         & \ge -c \alpha_{k} g_{k}^\top d_{k}                         &  & (\text{\cref{eq:relaxed_Armijo_condition}}) \notag                                         \\
         & \ge \frac{c \overline{\alpha}}{M+\mu_{k}} \norm{g_{k}}^2   &  & (\text{\cref{prop:backtracking_summary} and \cref{eq:BBounds_results_first}})\notag        \\
         & = \frac{c \overline{\alpha}}{M} \norm{g_{k}}^2.            &  & (\mu_{k} = 0 \text{ since } k \in K^0) \label{eq:lower_bound_overline_delta_plus_Delta_K0}
    \end{align}
    Using \cref{lem:error_delta_bounded_by_Delta}, we also have
    \begin{equation}
        \overline{\delta}_k + \Delta_{k} \le \delta_k + \frac{1+\feps}{1-\feps} \Delta_k + \Delta_k = \delta_k + \frac{2}{1-\feps} \Delta_{k}. \label{eq:upper_bound_overline_delta_plus_Delta_K0}
    \end{equation}
    Then, combining \cref{eq:lower_bound_overline_delta_plus_Delta_K0,eq:upper_bound_overline_delta_plus_Delta_K0} and summing over $k \in K^0$ gives
    \begin{equation*}
        \sum_{k \in K^0} \frac{c \overline{\alpha}}{M} \norm{g_{k}}^2   \leq \sum_{k \in K^0} \delta_{k} + \frac{2}{1-\feps} \sum_{k \in K^0} \Delta_{k}
        \leq \sum_{k \in K^0} \delta_{k} + \frac{2}{1-\feps} \qty(\overline{f}(x_0) - \overline{f}^* + \Delta_{\max}),
    \end{equation*}
    where the last inequality follows from \cref{lem:delta_bound_and_sum}.
    Rearranging this completes the proof.
    \myQED
\end{proof}

\subsection{Bound for the case \texorpdfstring{$\mu_k > 0$}{mu k > 0}}

Next, we show that the sum of $\delta_{k}$ over $k \in K^+$ is lower bounded.
Recall that $K^+$ is the set of iteration indices $k$ such that $\mu_k > 0$ as defined in \cref{eq:def_K0_Kplus}.
For $k\in K^+$, \cref{eq:offo_based_update} gives $\mu_k=\theta_k \sqrt{\varsigma + \sum_{j \in K^+_{k+1}} \norm{g_{j}}^2}$ with $\theta_k \in [\theta_{\min}, \theta_{\max}]$.

Before proceeding to the main proof, we present \cref{lem:sum_norm_g_over_mu_geq_2sqrt_leq_log}, which collects standard inequalities from the AdaGrad-Norm analysis~\citep[Lemma~3.2]{wardAdaGradStepsizesSharp2020}~\citep[Lemma~3.1]{Gratton08022024}.
Its proof is deferred to \cref{sec:proof_of_sum_norm_g_over_mu_geq_2sqrt_leq_log}.
\begin{lemma}\label{lem:sum_norm_g_over_mu_geq_2sqrt_leq_log}
    In \cref{alg:regularized_qn}, we have
    \begin{align*}
        \sum_{{k} \in K^+}\frac{\norm{g_{k}}^2}{\mu_{k}^2}
         & \le \frac{1}{\theta_{\min}^2} \qty(\ln\qty(\varsigma+\sum_{ k \in K^+} \norm{g_{k}}^2) - \ln(\varsigma)),  \\
        \sum_{{k} \in K^+} \frac{\norm{g_{k}}^2}{\mu_{k}}
         & \ge \frac{1}{\theta_{\max}} \qty(\sqrt{\varsigma + \sum_{{k} \in K^+} \norm{g_{k}}^2} - \sqrt{\varsigma}).
    \end{align*}
\end{lemma}

We now lower bound the sum of $\delta_k$ over $k \in K^+$ with $\norm{g_{k}}^2$.
We define the following constants for later use:
\begin{equation}\label{eq:def_C1_C2}
    C_1 \defeq \frac{\overline{\alpha}}{\theta_{\max}}, \qquad
    C_2 \defeq \frac{M\overline{\alpha}+L/2}{\theta_{\min}^2}.
\end{equation}
\begin{lemma}\label{lem:sum_small_delta_k_bound_Kplus}
    Under \cref{ass:LSmooth,ass:BBounds}, the sum of $\delta_k$ over $k \in K^+$ satisfies
    \begin{equation*}
        \sum_{k \in K^+} \delta_k \ge
        C_1 \qty(\sqrt{\varsigma + \sum_{k \in K^+} \norm{g_{k}}^2} - \sqrt{\varsigma})
        - C_2 \qty(\ln\qty(\varsigma+\sum_{k \in K^+} \norm{g_{k}}^2) - \ln(\varsigma)).
    \end{equation*}
\end{lemma}
\begin{proof}
    From the $L$-smoothness (\cref{ass:LSmooth}), for $k \in K^+$ we have
    \begin{align}
        \delta_k & = f(x_k) - f(x_k + \alpha_{k} d_{k})                                                          &  & (\text{\cref{eq:def_small_delta_k} and } x_{{k}+1}=x_{k} + \alpha_{k} d_{k})      \notag \\
                 & \ge -\alpha_{k} g_{k}^\top d_{k} - \frac{\alpha_{k}^2 L}{2} \norm{d_{k}}^2                    &  & (\text{\cref{ass:LSmooth}})      \notag                                                  \\
                 & \ge \qty(\frac{\alpha_{k}}{M+\mu_{k}} - \frac{\alpha_{k}^2 L}{2(m+\mu_{k})^2}) \norm{g_{k}}^2 &  & (\text{\cref{eq:BBounds_results_first,eq:BBounds_results_second}}) \notag                \\
                 & \ge \qty(\frac{\overline{\alpha}}{M+\mu_{k}} - \frac{L}{2(m+\mu_{k})^2}) \norm{g_{k}}^2.      &  & (\text{\cref{prop:backtracking_summary} and $\alpha_k \leq 1$})
        \label{eq:mu_positive_diff}
    \end{align}
    We can further bound the terms in \cref{eq:mu_positive_diff} as
    \begin{equation}\label{eq:mu_positive_diff_bound}
        \frac{\overline{\alpha}}{M+\mu_{k}} = \frac{\overline{\alpha}}{\mu_{k}} \frac{1}{1+\frac{M}{\mu_{k}}} \geq \frac{\overline{\alpha}}{\mu_{k}} \qty(1-\frac{M}{\mu_{k}}) = \frac{\overline{\alpha}}{\mu_{k}} - \frac{M\overline{\alpha}}{\mu_{k}^2}, \quad
        \frac{L}{2(m+\mu_{k})^2} \le \frac{L}{2\mu_{k}^2}.
    \end{equation}
    Thus, applying \cref{eq:mu_positive_diff_bound} to \cref{eq:mu_positive_diff} yields
    \begin{equation}\label{eq:mu_positive_diff_desired}
        \delta_k \ge \qty(\frac{\overline{\alpha}}{\mu_{k}} - \frac{M\overline{\alpha}+L/2}{\mu_{k}^2}) \norm{g_{k}}^2.
    \end{equation}
    Summing \cref{eq:mu_positive_diff_desired} over $k \in K^+$ and applying \cref{lem:sum_norm_g_over_mu_geq_2sqrt_leq_log} yield
    \begin{align*}
              & \sum_{{k} \in K^+} \delta_k \ge \sum_{{k} \in K^+} \qty(\frac{\overline{\alpha}}{\mu_{k}} - \frac{M\overline{\alpha}+L/2}{\mu_{k}^2})\norm{g_{k}}^2                                                                                                 \\
        \ge{} & \frac{\overline{\alpha}}{\theta_{\max}} \qty(\sqrt{\varsigma + \sum_{{k} \in K^+} \norm{g_{k}}^2} - \sqrt{\varsigma}) - \frac{M\overline{\alpha}+L/2}{\theta_{\min}^2} \qty(\ln\qty(\varsigma+\sum_{{k} \in K^+} \norm{g_{k}}^2) - \ln(\varsigma)),
    \end{align*}
    which yields the desired result.
    \myQED
\end{proof}

\subsection{Global Convergence Rate}

Finally, we combine the results for $k \in K^0$ stated in \cref{lem:sum_small_delta_k_bound_K0} and $k \in K^+$ stated in \cref{lem:sum_small_delta_k_bound_Kplus} to derive an upper bound on the total sum of $\norm{g_k}^2$.
We define a constant $F$ as follows:
\begin{equation}
    F \defeq f(x_0) - f^* + \frac{2}{1-\feps} \qty(\overline{f}(x_0) - \overline{f}^* + \Delta_{\max}) + C_1 \sqrt{\varsigma} - C_2 \ln(\varsigma).
    \label{eq:def_F}
\end{equation}
Then, we can obtain the following unified bound for the two types of iterations.
\begin{lemma}\label{lem:sum_norm_g_k_bound}
    Under \cref{ass:bounded,ass:LSmooth,ass:BBounds}, the sum of $\norm{g_{k}}^2$ over $k \in K=K^0\cup K^+$ satisfies
    \begin{equation*}
        F \ge
        \frac{c \overline{\alpha}}{M} \sum_{k \in K^0} \norm{g_{k}}^2 +
        C_1 \sqrt{\varsigma + \sum_{k \in K^+} \norm{g_{k}}^2} -
        C_2 \ln\qty(\varsigma+\sum_{k \in K^+} \norm{g_{k}}^2),
    \end{equation*}
\end{lemma}
\begin{proof}
    Since $f(x_{k_{\max}}) \ge f^*$ by \cref{ass:bounded} and $K=K^0 \cup K^+$, we have
    \begin{equation*}
        f(x_0) - f^* \geq f(x_0) - f(x_{k_{\max}}) = \sum_{k \in K} \delta_k = \sum_{k \in K^0} \delta_k + \sum_{k \in K^+} \delta_k.
    \end{equation*}
    From \cref{lem:sum_small_delta_k_bound_K0,lem:sum_small_delta_k_bound_Kplus}, we can further bound as
    \begin{multline*}
        f(x_0) - f^*
        \ge \frac{c \overline{\alpha}}{M} \sum_{k \in K^0} \norm{g_{k}}^2
        - \frac{2}{1-\feps} \qty(\overline{f}(x_0) - \overline{f}^* + \Delta_{\max})\\
        +C_1 \sqrt{\varsigma + \sum_{k \in K^+} \norm{g_{k}}^2}
        - C_1 \sqrt{\varsigma}
        - C_2 \ln\qty(\varsigma+\sum_{k \in K^+} \norm{g_{k}}^2)
        + C_2 \ln(\varsigma).
    \end{multline*}
    Substituting \cref{eq:def_F} yields the desired bound, concluding the proof.
    \myQED
\end{proof}

Now, we establish the following convergence rate for \cref{alg:regularized_qn}.
\begin{theorem}\label{thm:gc}
    Under the problem setting in \cref{sec:problem_setting} and \cref{ass:bounded,ass:LSmooth,ass:BBounds}, the average of squared gradient norms generated by \cref{alg:regularized_qn} satisfies
    \begin{equation*}
        \frac{1}{k_{\max}} \sum_{k=0}^{k_{\max}-1} \norm{g_k}^2 = \order{\frac{1}{k_{\max}}}.
    \end{equation*}
\end{theorem}
\begin{proof}
    Let us define the function $\phi\colon \mathbb{R}_{>0} \to \mathbb{R}$ as
    \begin{equation}\label{eq:phi_G}
        \phi(G) \defeq \frac{C_1}{2} \sqrt{G} - C_2 \ln(G).
    \end{equation}
    Since $\phi'(G)= \frac{C_1}{4 \sqrt{G}} - \frac{C_2}{G}$, its minimum over $G>0$ is attained at $G^\ast \defeq \qty(\frac{4C_2}{C_1})^2$ with
    \begin{equation*}
        \phi(G^\ast) = 2C_2 \qty(1 - \ln\qty(\frac{4C_2}{C_1})) = 2C_2 \ln\qty(\frac{\mathrm{e} C_1}{4C_2}),
    \end{equation*}
    where $\mathrm{e}$ is the base of the natural logarithm.
    Thus, \cref{lem:sum_norm_g_k_bound} yields
    \begin{align*}
        F
         & \ge
        \frac{c \overline{\alpha}}{M} \sum_{k \in K^0} \norm{g_{k}}^2 + \frac{C_1}{2} \sqrt{\sum_{k \in K^+} \norm{g_{k}}^2} + \phi\qty(\sum_{k \in K^+} \norm{g_{k}}^2 + \varsigma) \\
         & \ge \frac{c \overline{\alpha}}{M} \sum_{k \in K^0} \norm{g_{k}}^2 + \frac{C_1}{2} \sqrt{\sum_{k \in K^+} \norm{g_{k}}^2} + \phi(G^\ast),
    \end{align*}
    which is equivalent to
    \begin{equation}\label{eq:final_inequality_for_gc}
        0 \leq \frac{c \overline{\alpha}}{M} \sum_{k \in K^0} \norm{g_{k}}^2 + \frac{C_1}{2} \sqrt{\sum_{k \in K^+} \norm{g_{k}}^2} \leq F'
    \end{equation}
    where
    \begin{align}
        F' & \defeq F - \phi(G^\ast) \notag \\
           & = f(x_0) - f^* + \frac{2}{1-\feps} \qty(\overline{f}(x_0) - \overline{f}^* + \Delta_{\max}) + C_1 \sqrt{\varsigma} - C_2 \ln(\varsigma) - 2C_2 \ln\qty(\frac{\mathrm{e} C_1}{4C_2}). \label{eq:def_F_prime}
    \end{align}

    It remains to bound $\sum_{k\in K}\norm{g_k}^2 = \sum_{k\in K^0}\norm{g_k}^2 + \sum_{k\in K^+}\norm{g_k}^2$ given the constraint in \cref{eq:final_inequality_for_gc}.
    More generally, let $a,b>0$, $c\ge 0$ and $x,y\ge 0$, and consider the maximization problem $x+y$ subject to $a x + b\sqrt{y} \le c$.
    The maximum is achieved on the boundary $a x + b\sqrt{y}=c$ since $x+y$ is increasing in both variables.
    On this boundary, we can write $y=\left(\frac{c-a x}{b}\right)^2$, and thus $x+y$ is convex on an interval $0\le x\le c/a$, attaining its maximum at one of the endpoints.
    Evaluating at the endpoints yields the candidates $(x,y)=(c/a,0)$ and $(x,y)=(0,(c/b)^2)$, and therefore the maximum of $x+y$ is $\max\qty(c/a, (c/b)^2)$.
    Applying this result to \cref{eq:final_inequality_for_gc} gives
    \begin{align}
        \frac{1}{k_{\max}} \sum_{k=0}^{k_{\max}-1} \norm{g_k}^2
         & = \frac{1}{k_{\max}} \qty(\sum_{{k} \in K^0} \norm{g_{k}}^2 + \sum_{k \in K^+} \norm{g_{k}}^2)                    &  & (K=K^0 \cup K^+) \notag \\
         & \leq \frac{1}{k_{\max}} \max\qty(\frac{M F'}{c \overline{\alpha}}, \, \frac{4 F'^2}{C_1^2}), \label{eq:gc_result}
    \end{align}
    which completes the proof.
    \myQED
\end{proof}

As a remark, the function $\phi(G)$ in \cref{eq:phi_G} is closely related to the Lambert $W$ function~\citep{corlessLambertWFunction1996,Gratton08022024}.
For $x \in (-1/\mathrm{e},0)$, the equation $y \mathrm{e}^y = x$ admits two real negative solutions, and the smaller one is given by the second branch of the Lambert $W$ function, $y = W_{-1}(x) < 0$.
Using this function, the equation $\phi(G)=F$ can be solved in closed form, which leads to a sharper bound in \cref{thm:gc}.
Since this is not essential for the main argument, we omit the details and refer to~\citep{Gratton08022024}.

We next interpret \cref{eq:gc_result} in the proof of \cref{thm:gc}.
The contribution of the indices in $K^0$ is analogous to the classical analysis of gradient descent.
Indeed, for gradient descent with fixed step size $\alpha = 1/L$, one has
\begin{equation*}
    \frac{1}{k_{\max}}\sum_{k=0}^{k_{\max}-1} \norm{g_k}^2 \le \frac{1}{k_{\max}}\frac{2(f(x_0)-f^*)}{\alpha}.
\end{equation*}
Thus, the term associated with $K^0$ reflects the standard $\order{f(x_0)-f^*}$ dependence of first-order methods since $F'$ contains the initial gaps $f(x_0)-f^*$ and $\overline{f}(x_0)-\overline{f}^*$.
In contrast, the contribution of $K^+$ resembles the behavior of AdaGrad-Norm, whose convergence bound scales quadratically with the initial optimality gap~\citep{wardAdaGradStepsizesSharp2020}.

Finally, \cref{thm:gc} implies the iteration-complexity bound
\begin{equation*}
    \min\left\{k \relmiddle| \norm{\nabla f(x_k)} \leq \gtol \right\} = \order{\gtol^{-2}}.
\end{equation*}
\begin{MPCchange}
Let us examine the hidden constants in this bound.
By \cref{eq:gc_result}, there exists an index $0 \leq k < k_{\max}$ satisfying $\norm{\nabla f(x_k)} \leq \gtol$ whenever
\begin{equation}
    k_{\max}\geq \frac{1}{\gtol^2}
    \max\left\{\frac{MF'}{c\overline\alpha},\frac{4(F')^2}{C_1^2}\right\}. \label{eq:k_max_gtol_bound}
\end{equation}
In the following, we treat $c$, $\beta_{\min}$, $\theta_{\min}$, $\theta_{\max}$, and $\varsigma$ as fixed algorithmic parameters.
Recall that $mI \preceq B_k \preceq MI$ by \cref{ass:BBounds}.
Define the line-search conditioning factor $L_{\mathrm{ls}}\coloneqq\max\{1,L/m\}$.
Using \cref{eq:def_alpha_bar,eq:def_C1_C2,eq:def_Delta_max}, we obtain
\begin{align*}
    \overline\alpha & = \Theta\qty(\frac{1}{L_{\mathrm{ls}}}), \qquad
    C_1 = \Theta\qty(\frac{1}{L_{\mathrm{ls}}}), \qquad
    C_2 = \order{\frac{M}{L_{\mathrm{ls}}} + L}, \\
    \Delta_{\max}        & = \order{\frac{\feps \max\{1, \overline f(x_0), -\overline f^*\}}{1 - \feps}}.
\end{align*}
Here, $\Theta$ denotes upper and lower bounds up to positive constant factors.
Combining the logarithmic terms in \cref{eq:def_F_prime} and using the bounds above, we obtain
\begin{align*}
    C_1 \sqrt{\varsigma} - C_2 \ln\varsigma - 2C_2 \ln\qty(\frac{\mathrm{e} C_1}{4C_2}) &= C_1 \sqrt{\varsigma} + 2C_2 \qty(\ln\frac{4C_2}{C_1 \sqrt{\varsigma}} - 1) \\
    & = \order{\frac{1}{L_{\mathrm{ls}}} + \qty(\frac{M}{L_{\mathrm{ls}}} + L) \qty(1 + \ln(M + L L_{\mathrm{ls}}))}.
\end{align*}
Here, the additive constant in $1 + \ln(M + L L_{\mathrm{ls}})$ absorbs the fixed factors inside the logarithm.
We take $M \geq 1$ without loss of generality, so the logarithm is nonnegative.
Substituting this estimate and the order of $\Delta_{\max}$ into \cref{eq:def_F_prime} yields
\begin{align*}
    F' & = \mathcal{O}\bigg(f(x_0)-f^*+\frac{\overline f(x_0)-\overline f^*}{1-\feps}
    +\frac{\feps\max\{1,\overline f(x_0),-\overline f^*\}}{(1-\feps)^2} \\
    &\qquad{}+\frac{1}{L_{\mathrm{ls}}}+\qty(\frac{M}{L_{\mathrm{ls}}} + L)\qty(1+\ln(M+L L_{\mathrm{ls}}))\bigg).
\end{align*}
Substituting these bounds in \cref{eq:k_max_gtol_bound} gives the following upper bound on the number of iterations needed to reach this tolerance:
\begin{equation*}
    \order{\gtol^{-2}\max\{ML_{\mathrm{ls}}F',L_{\mathrm{ls}}^2(F')^2\}}.
\end{equation*}
For fixed algorithmic parameters, problem data, and $\feps$, all quantities in the maximum are independent of $\gtol$, giving the target-accuracy order $\order{\gtol^{-2}}$.
\end{MPCchange}

\section{Practical Choices of Algorithmic Variables}\label{sec:proposed_choice}

In \cref{sec:algorithm,sec:analysis}, we introduced \cref{alg:regularized_qn} and established its convergence properties.
In this section, we discuss practical choices of algorithmic variables and parameters that lead to fast and stable performance.
The quantities to be discussed here are the matrices $B_k$, the step sizes $\alpha_k$, and the regularization parameters $\mu_k$.

\subsection{Approximate Hessian \texorpdfstring{$B_k$}{B k}}\label{sec:proposed_choice_Bk}

We first discuss the construction of the matrices $B_k$ used in \refLine{line:d_k} of \cref{alg:regularized_qn}.
We show that the construction below satisfies \cref{ass:BBounds}.

\paragraph{Construction of $B_k$: }

We recall the L-BFGS construction~\citep{nocedal1999numerical,liuLimitedMemoryBFGS1989,berahasLimitedmemoryBFGSDisplacement2022}.
For an iteration $k$, we define
\begin{equation*}
    s_k \defeq x_{k+1}-x_k,\qquad
    y_k \defeq g_{k+1}-g_k
\end{equation*}
where $g_k$ and $g_{k+1}$ are the gradients at $x_k$ and $x_{k+1}$, respectively.
Starting from an initial matrix, an approximate Hessian is obtained by successively applying BFGS updates to a sequence of curvature pairs.
Let $\{(s_{k_i},y_{k_i})\}_{i=0}^{p-1}$ be a collection of curvature pairs ordered so that $k_0<\cdots<k_{p-1}$.
Starting from the initial matrix $B^{(0)} = \gamma I$ with $\gamma=\norm{y_{k_{p-1}}}^2 / y_{k_{p-1}}^\top s_{k_{p-1}}$, we define $\mathrm{BFGS}(\{(s_{k_i},y_{k_i})\}_{i=0}^{p-1})$ as the matrix $B^{(p)}$ obtained by sequentially applying the BFGS update as
\begin{equation*}
    B^{(j+1)} = B^{(j)} - \frac{B^{(j)} s_{k_j} s_{k_j}^\top B^{(j)}}{s_{k_j}^\top B^{(j)} s_{k_j}} + \frac{y_{k_j} y_{k_j}^\top}{y_{k_j}^\top s_{k_j}}
\end{equation*}
for $j=0,1,\ldots,p-1$, where each curvature pair $(s_{k_j},y_{k_j})$ is applied in order.
Here, $p>0$ denotes the maximum number of stored curvature pairs and is a fixed small integer in the L-BFGS method.
When fewer than $p$ curvature pairs are available, the BFGS update is applied using all currently available pairs.

We now give a sufficient condition for the upper spectral bound in \cref{ass:BBounds}, based on variants of existing results~\citep[Theorem~5.5]{gaoQuasiNewtonMethodsSuperlinear2018} \citep[Lemma~1]{gowerStochasticBlockBFGS2016}.
Its proof is deferred to \cref{app:proof_of_bounded_BFGS}.
\begin{proposition}\label{prop:bounded_BFGS}
    Let $\{(s_{k_i},y_{k_i})\}_{i=0}^{p-1}$ be curvature pairs with $s_{k_i} \neq 0$.
    \MPCchanged{Suppose that there exists a constant $\Lambda>0$ such that every pair $(s,y)$ in this collection satisfies}
    \begin{equation}
        \MPCchanged{y^\top s \ge \frac{1}{\Lambda}\norm{y}^2>0}. \label{eq:curvature_conditions1}
    \end{equation}
    \MPCchanged{Then}
    \begin{equation}\label{eq:bounded_BFGS_mM}
        \MPCchanged{0\prec \mathrm{BFGS}(\{s_{k_i},y_{k_i}\}_{i=0}^{p-1}) \preceq (1+p)\Lambda I}.
    \end{equation}
\end{proposition}
After the modified secant correction described below, we apply scalar Powell-type damping to a candidate pair with $y_{k_i}^{\top}s_{k_i}<0$~\citep{powellAlgorithmsNonlinearConstraints1978,lotfiStochasticDampedLBFGS2020}.
For a parameter $\eta_{\mathrm{damp}}\in(0,1)$, let $b_i=\norm{y_{k_i}}^2/\abs{y_{k_i}^{\top}s_{k_i}}$ and set
\begin{equation*}
    \overline y_{k_i}
    = \theta_i^{\mathrm{damp}} y_{k_i}
    + (1 - \theta_i^{\mathrm{damp}}) b_i s_{k_i},
    \qquad
    \theta_i^{\mathrm{damp}}
    = \frac{\eta_{\mathrm{damp}} b_i \norm{s_{k_i}}^2}{b_i \norm{s_{k_i}}^2 - y_{k_i}^{\top} s_{k_i}} \in [0, 1).
\end{equation*}
If $y_{k_i}^{\top}s_{k_i}\geq0$, we set $\overline y_{k_i}=y_{k_i}$.
Candidate pairs proceed to cautious selection only when the damping can be computed safely.

We then select the damped pairs using a cautious update in the spirit of~\citep{liGlobalConvergenceBFGS2001,mannelStructuredLBFGSMethod2024,mannelGlobalizationLBFGSBarzilai2025}.
\begin{MPCchange}
Define the fixed initial scale $\rho_0$ and a positive constant $\omega_{\min}$ by
\begin{equation*}
    \rho_0\defeq\max\{1,\norm{g_0}\},
    \qquad
    \omega_{\min}\defeq\min\{1,\sqrt{\varsigma}\}>0.
\end{equation*}
For a fixed $\Lambda\geq1$, a candidate pair $(s,\overline y)$ is retained only if
\begin{equation}\label{eq:implemented_cautious_test}
    \overline y^\top s\geq\frac{\norm{\overline y}^2}{\Lambda\rho_0}>0.
\end{equation}
Consequently, every retained pair $(s,\overline y/\rho_0)$ satisfies \cref{eq:curvature_conditions1}, and the definition of the BFGS update gives
\begin{equation}\label{eq:BFGS_scaling}
    \mathrm{BFGS}(\{(s_{k_i},\overline y_{k_i})\})
    =\rho_0\mathrm{BFGS}(\{(s_{k_i},\overline y_{k_i}/\rho_0)\}).
\end{equation}
To supply the lower spectral bound in \cref{ass:BBounds}, we add the fixed positive diagonal matrix $\omega_{\min}I$ to the L-BFGS matrix.
Using at most $p$ retained pairs, we define
\begin{equation}\label{eq:proposed_choice_of_Bk}
    B_k=
    \begin{cases}
        \omega_{\min}I+\mathrm{BFGS}(\{(s_{k_i},\overline y_{k_i})\}), & \text{if the memory is nonempty}, \\
        (\omega_{\min}+\rho_0)I,                                      & \text{otherwise}.
    \end{cases}
\end{equation}
It follows from \cref{prop:bounded_BFGS,eq:BFGS_scaling}, and $\omega_{\min}\leq1\leq\rho_0$ that
\begin{equation*}
    \omega_{\min}I\preceq B_k\preceq\qty(1+(1+p)\Lambda)\rho_0I.
\end{equation*}
Thus, \cref{ass:BBounds} holds with the iteration-independent constants $m=\omega_{\min}$ and $M=\qty(1+(1+p)\Lambda)\rho_0$.
\end{MPCchange}
In our implementation, we use $p=10$ and $\eta_{\mathrm{damp}}=0.8$.
\MPCchanged{We set $\Lambda=10^8$ and retain a pair only if it also satisfies $\overline y^\top s \geq 10^{-16}$, to reject numerically degenerate pairs.}

\begin{MPCchange}
\paragraph{Exact Compact Solve: }

We use a compact representation of the L-BFGS matrix in \cref{eq:proposed_choice_of_Bk} to solve the regularized system exactly~\citep{kanzowRegularizationLimitedMemory2023}, as analyzed in \cref{sec:analysis}.
If the compact factorization is unavailable or does not return a finite descent direction, the implementation falls back to the matrix used when the memory is empty, which also satisfies \cref{ass:BBounds}.
\end{MPCchange}

\paragraph{Shifted-Pair Approximation: }

For comparison, the implementation retains the simpler shifted-pair approximation proposed by Kanzow and Steck~\citep{kanzowRegularizationLimitedMemory2023}
\begin{equation}\label{eq:approx_computation_of_regularized_quasi_Newton}
    B_k+\mu_k I
    \approx\mathrm{BFGS}(\{s_{k_i},\overline y_{k_i}+(\mu_k+\omega_{\min})s_{k_i}\}_{i=0}^{p-1}).
\end{equation}
Only \texttt{Ours-AR} uses this two-loop approximation, together with the optional restart in \cref{sec:proposed_choice_mu}.
The analysis applies to the exact compact solve used by \texttt{Ours}.

\begin{MPCchange}
\paragraph{Modified Secant Correction: }

Both proposed variants apply a function-based modified secant correction when the observed decrease is sufficiently large relative to the function-value error bound, before damping and cautious selection~\citep{zhangNewQuasiNewtonEquation1999,yuanConvergenceAnalysisModified2010}.
\end{MPCchange}
A scale-aware safeguard suppresses the correction whenever it would dominate the raw gradient difference.
\MPCchanged{The final corrected and damped pair must still satisfy \cref{eq:implemented_cautious_test}, so this modification preserves the spectral guarantee for \texttt{Ours}.}
The correction formula and safeguard are specified in our open-source implementation.

\subsection{Step Size \texorpdfstring{$\alpha_k$}{alpha k}}\label{sec:proposed_choice_alpha}

We next discuss the adjustment of the step size $\alpha_k$ in \refLine{line:ls_backtracking} of \cref{alg:backtracking_linesearch}.
As indicated in the pseudo-code, we choose $\alpha_k$ using interpolation of the objective function.
Mor\'e--Thuente line searches~\citep{10.1145/192115.192132} are widely used in quasi-Newton methods~\citep{2020SciPy-NMeth}, and they typically employ cubic or quadratic interpolation to select trial step sizes.
In our implementation, we use an essentially identical procedure, except that the trial step size is obtained by safeguarding it within the interval $[\beta_{\min}\alpha_k, \beta_{\max}\alpha_k]$.
For all numerical experiments, we set the Armijo parameter to $c = 10^{-4}$ and choose the interpolation bounds as $\beta_{\min} = 1/16$ and $\beta_{\max} = 15/16$.

Furthermore, when $\mu_k>0$ and $\alpha_k=1$, we introduce a one-time adjustment of the trial step size using gradient information.
Let $d_k$ be the search direction and $g_k$ and $g_{\mathrm{try}}$ the gradients at the current and trial points.
If the directional derivative along $d_k$ changes sign, $d_k^\top g_k<0$ and $d_k^\top g_{\mathrm{try}}>0$, and the trial gradient is sufficiently aligned with $d_k$, namely $d_k^\top g_{\mathrm{try}}>0.5\norm{d_k}\norm{g_{\mathrm{try}}}$, we rescale $\alpha_k$ once by a secant-like factor,
\begin{equation*}
    \alpha_k \gets \mathrm{clip}\qty(\alpha_k \frac{-d_k^\top g_k}{d_k^\top g_{\mathrm{try}}-d_k^\top g_k},\beta_{\min}\alpha_k,\beta_{\max}\alpha_k),
\end{equation*}
where $\mathrm{clip}(x,a,b) \defeq \min(\max(x,a),b)$ is a clipping function.
When objective values are unreliable and backtracking and interpolation do not occur, this correction effectively refines the step size.
After the adjustment, the relaxed Armijo condition is evaluated again, and ordinary backtracking continues if necessary.
The clipping interval gives $\alpha_k\geq\beta_{\min}$ whenever this adjustment is applied, so the conservative lower bound $\overline\alpha$ in \cref{prop:backtracking_summary} and the subsequent convergence analysis remain valid.

\subsection{Regularization Parameter \texorpdfstring{$\mu_k$}{mu k}}\label{sec:proposed_choice_mu}

We finally discuss the adjustment of the regularization parameter $\mu_k$ in \refLine{line:update_theta} of \cref{alg:regularized_qn}.
\paragraph{Regularization Update: }

The OFFO accumulator contains exactly the gradient norms generated during regularized iterations:
\begin{equation*}
    G_k^2=\varsigma+\sum_{j\in K^+_{k+1}}\norm{g_j}^2.
\end{equation*}
For a parameter $\eta_\mu>0$, we set
\begin{equation*}
    \mu_k=\mathrm{clip}\left(\eta_\mu\norm{g_k},\theta_{\min}G_k,\theta_{\max}G_k\right).
\end{equation*}
This clipping rule gives $\mu_k=\theta_kG_k$ for some $\theta_k\in[\theta_{\min},\theta_{\max}]$, as required by \cref{eq:offo_based_update}.
In our implementation, $\eta_\mu=1/10$, $\theta_{\min}=1/100$, $\theta_{\max}=1$, and \MPCchanged{$\varsigma=10^{-20}$, so the initial accumulator scale is $G_{-1}=\sqrt{\varsigma}=10^{-10}$}.

\paragraph{Optional Accumulator Restart: }

The optional restart is a practical heuristic intended to reduce over-regularization after clear progress in the objective value.
For a fixed threshold $\tau>0$, the heuristic restarts the accumulator when the method returns from $K^+$ to $K^0$ and
\begin{equation}\label{eq:restart_condition}
    \min_{j \in K_k^0} \qty(\overline{f}(x_j)-\Delta_j)-\overline{f}(x_k) \geq \tau.
\end{equation}
At a restart in iteration $k$, the squared accumulator is reset to $\varsigma$, discarding all gradient-norm contributions from regularized iterations before $k$.
Subsequent regularized iterations accumulate only gradient norms generated after $k$.
The global classification sets $K^0$ and $K^+$ and the reference values in \cref{eq:restart_condition} are retained.
The numerical experiments distinguish the no-restart and restarted variants.
The former coincides with the analyzed accumulator, whereas the latter uses $\tau=1$ and at most $10$ restarts.

\section{Numerical Experiments}\label{sec:experiments}

We conducted numerical experiments to evaluate the practical performance of our proposed method and to compare it with existing optimization algorithms.

\subsection{Methods}\label{sssec:methods}

We compared the following methods in our experiments.

\begin{itemize}
    \item (\texttt{Ours}): Our proposed regularized quasi-Newton method
          \begin{itemize}
              \item We used \cref{alg:regularized_qn,alg:backtracking_linesearch} with the details in \cref{sec:proposed_choice}.
              \MPCchanged{\item Although \cref{alg:regularized_qn} uses exact gradients under the problem setting in \cref{sec:problem_setting}, we also supplied gradients with additive noise or rounding errors to examine stability beyond this setting.}
              \item We used the modified secant condition described in \cref{sec:proposed_choice_Bk}.
              \MPCchanged{\item (\texttt{Ours-AR}) denotes the variant with shifted-pair approximation (A) and accumulator restart (R), as described in \cref{eq:approx_computation_of_regularized_quasi_Newton,sec:proposed_choice_mu}.}
          \end{itemize}

    \item (\texttt{Line}): Line search L-BFGS method
          \begin{itemize}
              \item We used our Python implementation ported from libLBFGS~\citep{Okazaki2007libLBFGS}.
              \item The step size was determined using the Mor\'e--Thuente line search, closely resembling SciPy's implementation.
          \end{itemize}

    \item (\texttt{Reg}): Regularized L-BFGS method~\citep{kanzowRegularizationLimitedMemory2023}
          \begin{itemize}
              \item A quadratic regularization-based quasi-Newton method that does not rely on line search and is conceptually closer to trust-region methods.
              \item We wrapped the function \texttt{regLBFGS} with \texttt{solveNonmonotone} from~\citep{kanzowRegularizationLimitedMemory2023}.
              \item (\texttt{Reg-Sec}) denotes the variant using the approximate computation in \cref{eq:approx_computation_of_regularized_quasi_Newton}.
              It wraps the \texttt{regLBFGSsec} function from~\citep{kanzowRegularizationLimitedMemory2023}.
          \end{itemize}

    \item (\texttt{SciPy}): SciPy's L-BFGS-B method~\citep{2020SciPy-NMeth}
          \begin{itemize}
              \item An established L-BFGS method implemented in C and called from Python.
              \item We used the default parameters, with the exception that \texttt{ftol} was set to $0$ to avoid premature termination.
          \end{itemize}

    \item (\texttt{NTQN}): Noise-tolerant quasi-Newton method~\citep{xieAnalysisBFGSMethod2020,shiNoiseTolerantQuasiNewtonAlgorithm2022}
          \begin{itemize}
              \item A method designed to accommodate inexact gradient information, as mentioned in \cref{sec:problem_setting}.
              \MPCchanged{\item For the primary comparison, instead of their recommended termination mode (\texttt{terminate=3}), we used \texttt{terminate=1} with an additional stop at the common gradient tolerance.
              This allows NTQN to continue toward the common gradient tolerance rather than stopping at its estimated noise level.}
          \end{itemize}

    \MPCchanged{\item (\texttt{ASTR1-Adagrad}): The purely first-order \texttt{adagrad} variant of ASTR1.
    \begin{itemize}
        \item The parameter setting is the same as reported by Gratton et al.~\citep[Table~1]{Gratton08022024}.
    \end{itemize}}
\end{itemize}

In the following, each method is referred to by the name given in parentheses.
Except for the first-order \texttt{ASTR1-Adagrad} baseline, all these methods are L-BFGS-based, and we set the number of stored curvature pairs to $10$, the default value in SciPy's L-BFGS-B implementation.
All experiments were conducted on an Apple Mac Pro (Late 2013) with a 2.7~GHz 12-core Intel Xeon E5-2697 v2 processor and 64~GB of memory. 
All computations were performed using the CPU only.

\subsection{Experimental Results}\label{sec:results}

\subsubsection{Test Problems and Settings}\label{sssec:problems}

We evaluated all algorithms on unconstrained problems from the CUTEst benchmark collection~\citep{gouldCUTEstConstrainedUnconstrained2015}, implemented in Python using the PyCUTEst interface~\citep{PyCUTEst2022}.
All problems were initialized with the default starting points provided by the library.
Following previous work~\citep{kanzowRegularizationLimitedMemory2023}, we excluded problems whose initial points were already stationary or whose function values or gradients contained \texttt{NaN} or \texttt{INF}.
\MPCchanged{The retained problems have dimensions ranging from $2$ to $123{,}200$.}

We organize the CUTEst experiments into the following two groups:
\begin{itemize}
    \item \emph{Artificial-noise experiments} on 220 problems using 64-bit arithmetic (results in \cref{sssec:noise_results})
          \begin{itemize}
              \MPCchanged{\item \emph{Function-value noise}: We add independent noise $\operatorname{unif}(-10^{-3}, 10^{-3})$ to each function evaluation while retaining exact gradients.
              We set the nominal error allowance to $\feps = 10^{-2}$.}
              \item \emph{Joint noise}: We additionally perturb each gradient component independently using the same distribution.
              We set $\feps = 10^{-2}$.
              \MPCchanged{\item \emph{Error-bound sensitivity}: Using the function-value-noise setting, we compare the underestimated allowance $\feps = 10^{-4}$, the nominal allowance $\feps = 10^{-2}$, and the overestimated allowance $\feps = 10^{-1}$.}
          \end{itemize}

    \item \emph{Floating-point-precision experiments} (results in \cref{sssec:exact_results})
          \begin{itemize}
              \item \emph{64-bit evaluations}: 220 problems with $\feps = 2.22 \times 10^{-9}$.
              \item \emph{32-bit evaluations}: 220 problems with $\feps = 1.19 \times 10^{-3}$.
              \item \emph{16-bit evaluations}: 152 problems with $\feps = 9.77 \times 10^{-2}$.
          \end{itemize}
\end{itemize}
Here, $\mathrm{unif}(a, b)$ denotes a uniform random variable in the interval $[a, b]$.
The additive uniform-noise model follows the model used by Shi et al.~\citep{shiNoiseTolerantQuasiNewtonAlgorithm2022}.
\MPCchanged{Each artificial-noise setting uses the five seeds $0, 1, 2, 3, 4$.}
\MPCchanged{Each problem--seed pair is treated as a separate test instance in the performance profiles.}
For NTQN in the function-value-noise and joint-noise experiments, we supplied the absolute noise bound $\texttt{eps\_f} = 10^{-2}$, which conservatively bounds function perturbations whose magnitude is at most $10^{-3}$.
\MPCchanged{We supplied componentwise gradient-error bounds of $0$ and $10^{-3}$, respectively, through \texttt{eps\_g}.}

In the reduced precision settings, we first cast the input vector $x$ to the target precision, evaluate CUTEst in 64-bit precision at the rounded input, and then round the returned function value and gradient to the target precision.
The solver receives these rounded quantities as 64-bit values, so this tests reduced-precision input and output rather than entirely low-precision arithmetic.

The error allowances for the three precision settings are summarized in \cref{tab:machine_epsilon}.
\MPCchanged{These values are empirical safety margins for accumulated floating-point errors rather than certified bounds.
The multipliers of machine epsilon are $10^7$, $10^4$, and $10^2$ for 64-bit, 32-bit, and 16-bit precision, respectively.
Because the appropriate value of $\feps$ may not be known accurately in practice, the sensitivity experiment examines how under- and overestimation affect performance.}
We also set the maximum number of iterations to $k_{\max} = 15{,}000$ and timed out runs exceeding 10 minutes per problem.

\begin{table}[t]
    \centering
    \caption{
        The ``min abs value'' is the smallest positive normalized number, ``machine epsilon'' is the spacing from $1$ to the next larger representable number, and $\feps$ is the parameter in \eqref{eq:error_model} used in our experiments.
    }
    \begin{tabular}{l|ccc}
        \toprule
        explanation             & min abs value           & machine epsilon                        & $\feps$                  \\
        \midrule
        64-bit double precision & $2.23 \times 10^{-308}$ & $2^{-52} \approx 2.22 \times 10^{-16}$ & $2.22 \times 10^{-9}$ \\
        32-bit single precision & $1.18 \times 10^{-38}$ & $2^{-23} \approx 1.19 \times 10^{-7}$ & $1.19 \times 10^{-3}$ \\
        16-bit half precision   & $6.10 \times 10^{-5}$ & $2^{-10} \approx 9.77 \times 10^{-4}$ & $9.77 \times 10^{-2}$ \\
        \bottomrule
    \end{tabular}
    \label{tab:machine_epsilon}
\end{table}

\subsubsection{Performance Profiles}\label{sssec:metrics}

We used performance profiles~\citep{dolanBenchmarkingOptimizationSoftware2002} to compare the algorithms.
Let $P$ denote the set of test instances and $S$ the set of solvers.
For each instance $p \in P$ and solver $s \in S$, let $n_{p,s} > 0$ denote the number of oracle calls (i.e., function and gradient evaluations) required to reach a specified gradient norm tolerance $\gtol$.
In this experiment, we use the number of oracle calls rather than clock time to avoid the influence of algorithmic overhead and to focus on oracle efficiency.
When a solver does not reach the gradient tolerance before termination or the iteration or time limit, we set $n_{p,s} = +\infty$.
Then, the performance ratio $r_{p,s}$ is defined as
\begin{equation*}
    r_{p,s} = \frac{n_{p,s}}{\min_{s' \in S} n_{p,s'}}.
\end{equation*}
If no solver succeeds on an instance, we set all its performance ratios to $+\infty$.
For a given threshold $\tau \geq 1$ and solver $s$, the performance profile plots the proportion of test instances for which $r_{p,s} \leq \tau$.
Thus, the $y$-coordinate for solver $s$ at $x$-coordinate $\tau$ is given by
\begin{equation*}
    \rho_s(\tau) = \frac{1}{\abs{P}} \abs{\left\{ p \in P \relmiddle| r_{p,s} \leq \tau \right\}}.
\end{equation*}
A curve that lies above others indicates better performance.
Following standard implementations~\citep{2020SciPy-NMeth}, we used the infinity norm for the gradient, though other norms may equally be employed.
All reported success percentages use the same common gradient tolerance as the corresponding profiles.
\MPCchanged{For joint noise and reduced precision, success is assessed using the gradient returned by the oracle.}
A run that terminates without attaining that tolerance is counted as unsolved for this metric.

\subsubsection{Results With Artificial Noise}\label{sssec:noise_results}

We first present results for the artificial-noise settings described in \cref{sssec:problems}.
\MPCchanged{\Cref{fig:results_function_noise} shows the performance profile for the function-value-noise experiment, which follows the problem setting in \cref{sec:problem_setting}.}
\Cref{fig:results_noise} shows the performance profile for the joint-noise experiment.
\MPCchanged{In the theory-aligned function-value-noise setting, both proposed variants attain higher solved percentages than the existing quasi-Newton baselines.}
The same overall tendency holds in the joint-noise setting.
Both proposed variants solve at least one instance of dimension $123{,}200$, the largest dimension in the retained collection.
\MPCchanged{The joint-noise result provides complementary empirical evidence beyond the problem setting in \cref{sec:problem_setting}.}

\begin{MPCchange}
We also examine how NTQN's recommended noise-dependent termination rule affects the comparison in the function-value-noise experiment.
With NTQN's recommended termination rule, $50.1\%$ of the instances attain the gradient tolerance $\gtol$, compared with $81.2\%$ under the termination settings used in our primary comparison, which include a stop at $\gtol$.
All recommended-mode runs terminate without timeout or exception, and many stop when function-value changes reach the estimated noise level.
The lower scored fraction therefore mainly reflects NTQN's noise-dependent stopping rule rather than failed execution.
Thus, our primary comparison gives NTQN a greater opportunity to attain the evaluation tolerance, avoiding a disadvantage due to its noise-dependent termination rule.
\end{MPCchange}

\begin{MPCchange}
\Cref{fig:results_error_bound_sensitivity} examines sensitivity to the supplied error bound by comparing both variants under all three tested bounds.
The nominal bound gives the highest solved percentage for \texttt{Ours}, whereas \texttt{Ours-AR} is less sensitive to overestimation.
Overestimating the error bound triggers OFFO-inspired regularization earlier, and the resulting conservative steps are the main cause of the lower solved percentage.
Restart resets the accumulated regularization scale after sufficient progress, mitigating subsequent over-regularization.
The improved performance of \texttt{Ours-AR} indicates the effectiveness of this restart mechanism.
Reducing the allowance from the conservative nominal value to the tested underestimated value has only a limited effect on the solved percentages in this experiment.
\end{MPCchange}

\begin{MPCchange}
\begin{figure}[tp]
    \ifdefined\MPCRevision\captionsetup{font={color=red}}\fi
    \centering
    \begin{subfigure}[b]{\resultsNoiseMainWidth\textwidth}
        \centering
        \includegraphics[width=\textwidth]{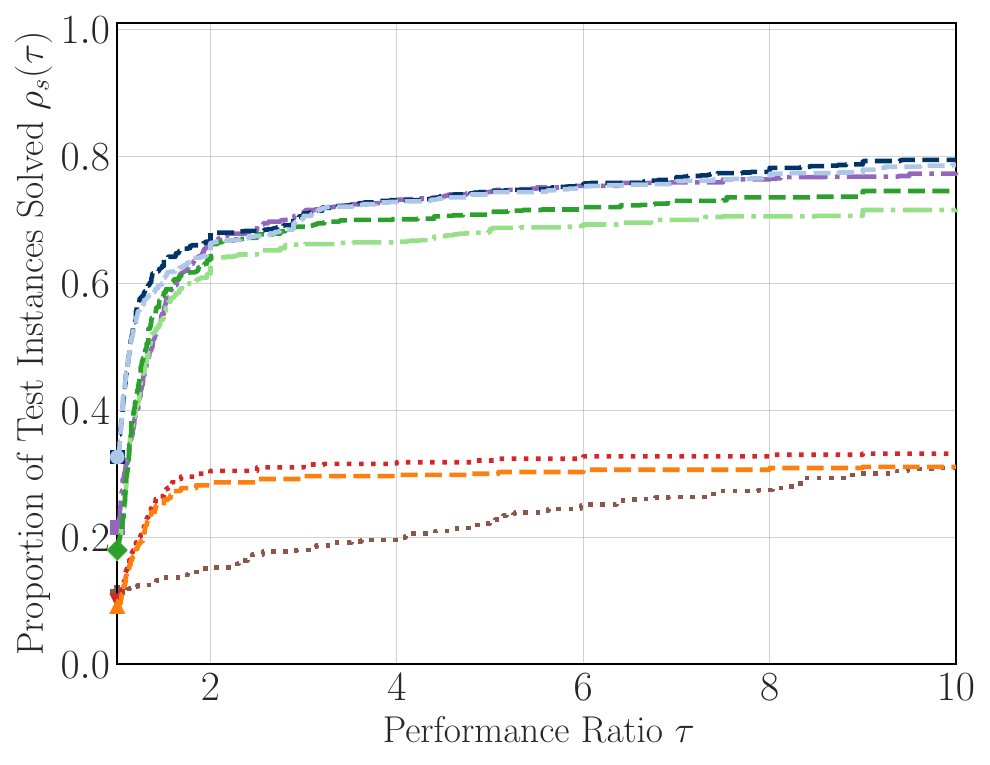}
    \end{subfigure}
    \begin{subfigure}[b]{\resultsNoiseLegendWidth\textwidth}
        \centering
        \includegraphics[width=\textwidth]{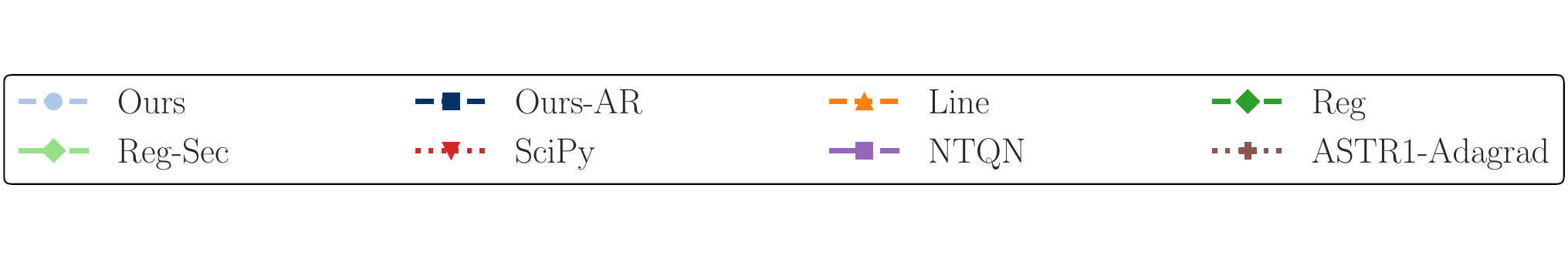}
    \end{subfigure}
    \caption{Performance profiles under function-value noise of level $10^{-3}$, with exact gradients, $\feps = 10^{-2}$, and $\gtol = 10^{-2}$.}
    \label{fig:results_function_noise}
\end{figure}
\end{MPCchange}

\begin{figure}[tp]
    \centering
    \begin{subfigure}[b]{\resultsNoiseMainWidth\textwidth}
        \centering
        \includegraphics[width=\textwidth]{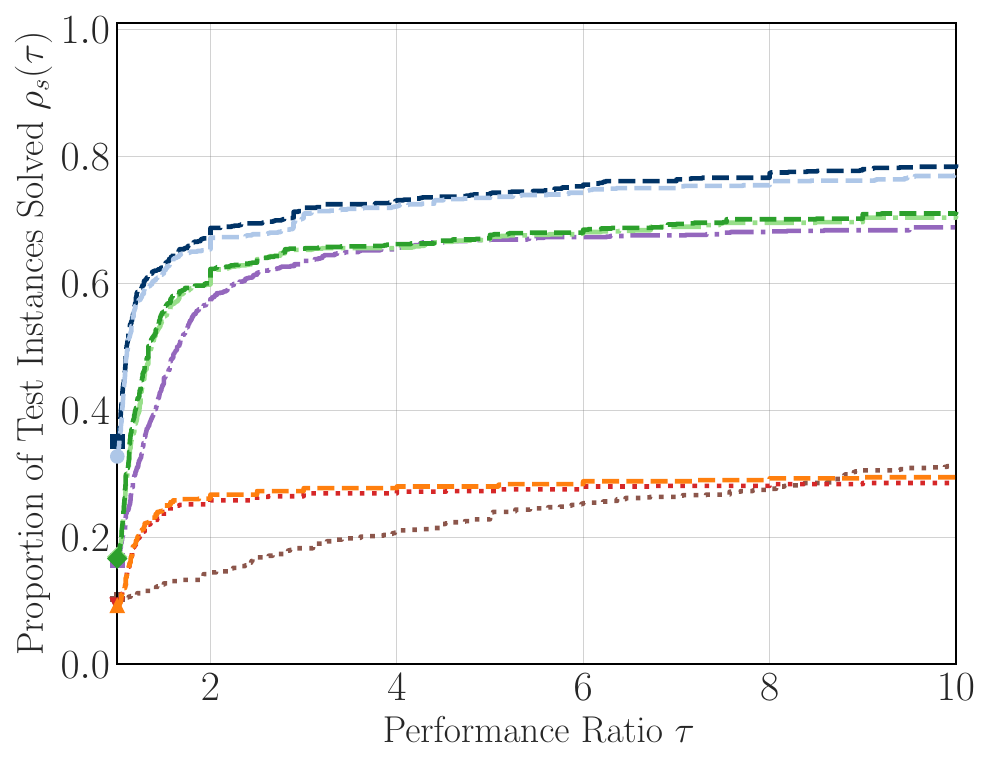}
    \end{subfigure}
    \begin{subfigure}[b]{\resultsNoiseLegendWidth\textwidth}
        \centering
        \includegraphics[width=\textwidth]{imgs/compare/_legend_noise.pdf}
    \end{subfigure}
    \caption{Performance profiles under joint function-and-gradient noise of componentwise level $10^{-3}$, with $\gtol = 10^{-2}$.}
    \label{fig:results_noise}
\end{figure}

\begin{MPCchange}
\begin{figure}[tp]
    \ifdefined\MPCRevision\captionsetup{font={color=red}}\fi
    \centering
    \begin{subfigure}[b]{\resultsNoiseMainWidth\textwidth}
        \centering
        \includegraphics[width=\textwidth]{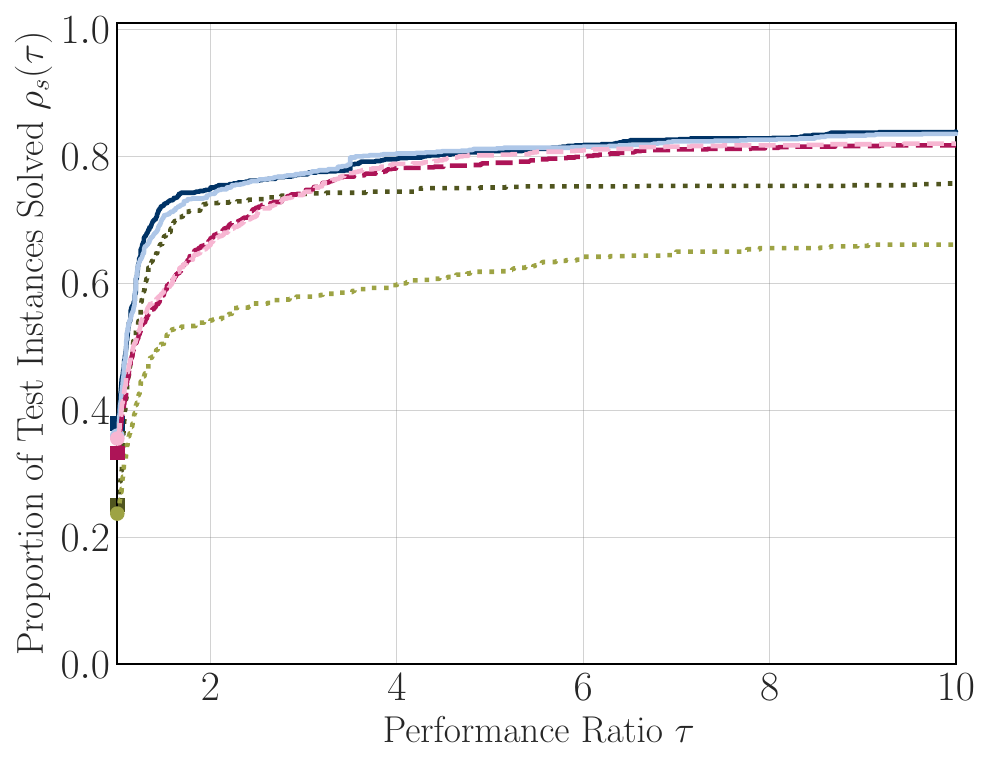}
    \end{subfigure}
    \begin{subfigure}[b]{\resultsNoiseLegendWidth\textwidth}
        \centering
        \includegraphics[width=\textwidth]{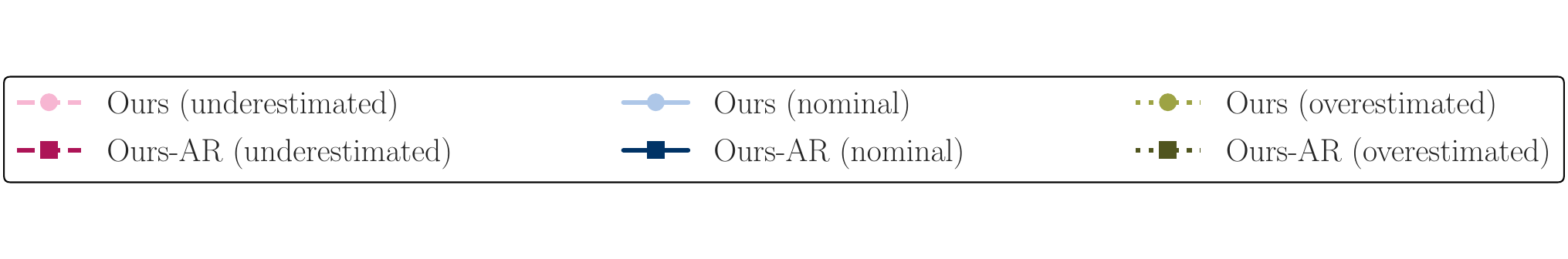}
    \end{subfigure}
    \caption{Performance profiles for error-bound misspecification under function-value noise of level $10^{-3}$, with $\gtol = 10^{-2}$.}
    \label{fig:results_error_bound_sensitivity}
\end{figure}
\end{MPCchange}

\subsubsection{Results at Different Precision Levels}\label{sssec:exact_results}

We next present performance profiles for each floating-point precision (64-bit, 32-bit, and 16-bit).
We evaluated each method at three gradient norm tolerance levels, i.e., $\gtol \in \{10^{-1}, 10^{-3}, 10^{-5}\}$.
These tolerance levels represent different optimization regimes, from moderate accuracy to high precision requirements.
\Cref{fig:results_all_precisions} shows the performance profiles across all three precision levels.
These performance profiles quantify how attained accuracy changes as oracle precision is reduced, while both proposed variants attain higher solved percentages than all existing quasi-Newton baselines at 16-bit and 32-bit precision.
\MPCchanged{Across the artificial-noise and reduced-precision experiments, the purely first-order \texttt{ASTR1-Adagrad} baseline solves fewer instances and generally requires more oracle calls than the proposed variants.}
These results support the practical motivation for invoking the OFFO-based scale selectively while exploiting quasi-Newton steps.

\begin{figure}[t]
    \centering
    \begin{subfigure}[b]{\resultsAllPrecisionsSubfigWidth\textwidth}
        \centering
        \includegraphics[width=\textwidth]{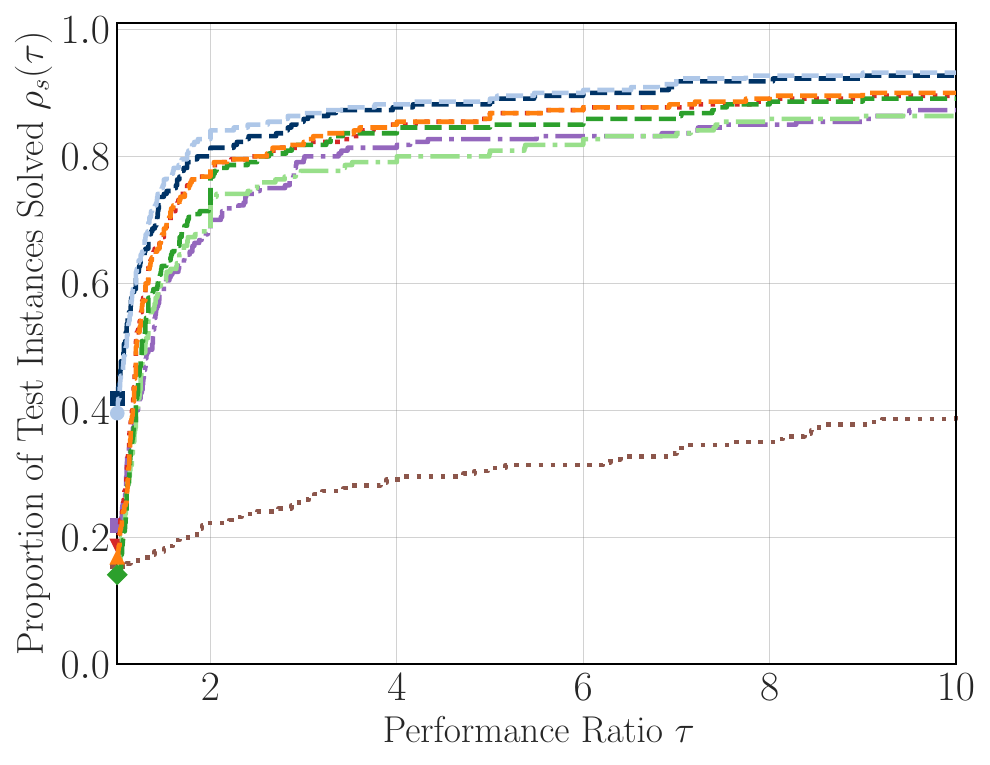}
        \caption{64-bit, $\gtol=10^{-1}$}
    \end{subfigure}
    \hfill
    \begin{subfigure}[b]{\resultsAllPrecisionsSubfigWidth\textwidth}
        \centering
        \includegraphics[width=\textwidth]{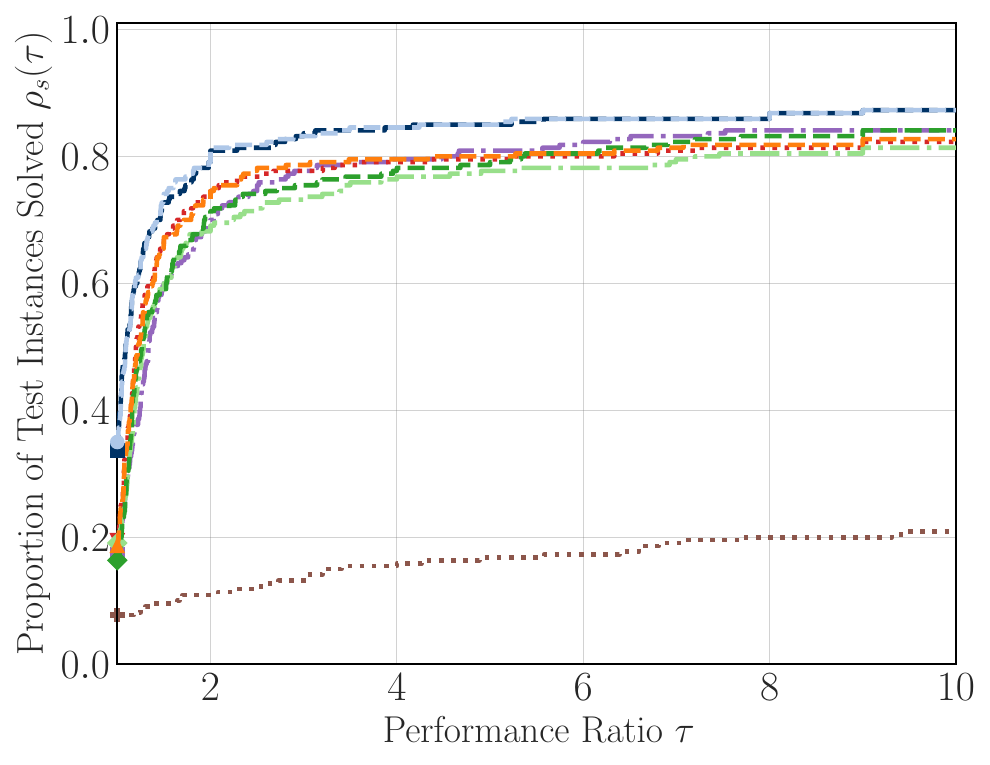}
        \caption{64-bit, $\gtol=10^{-3}$}
    \end{subfigure}
    \hfill
    \begin{subfigure}[b]{\resultsAllPrecisionsSubfigWidth\textwidth}
        \centering
        \includegraphics[width=\textwidth]{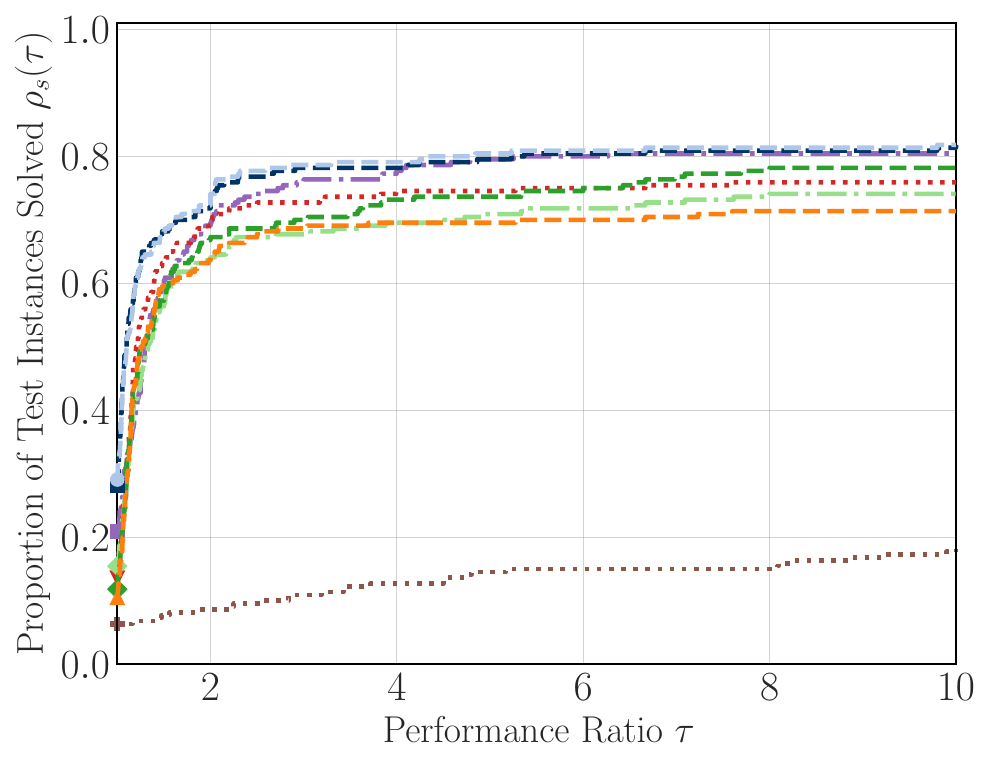}
        \caption{64-bit, $\gtol=10^{-5}$}
    \end{subfigure}

    \vspace{0.5em}

    \begin{subfigure}[b]{\resultsAllPrecisionsSubfigWidth\textwidth}
        \centering
        \includegraphics[width=\textwidth]{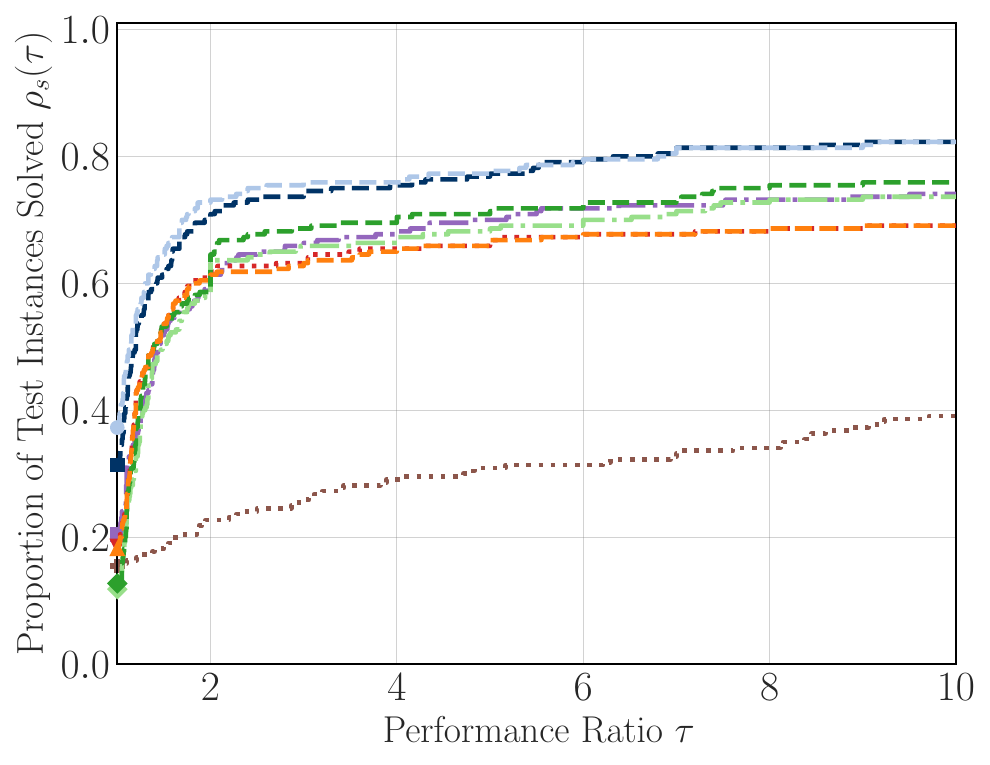}
        \caption{32-bit, $\gtol=10^{-1}$}
    \end{subfigure}
    \hfill
    \begin{subfigure}[b]{\resultsAllPrecisionsSubfigWidth\textwidth}
        \centering
        \includegraphics[width=\textwidth]{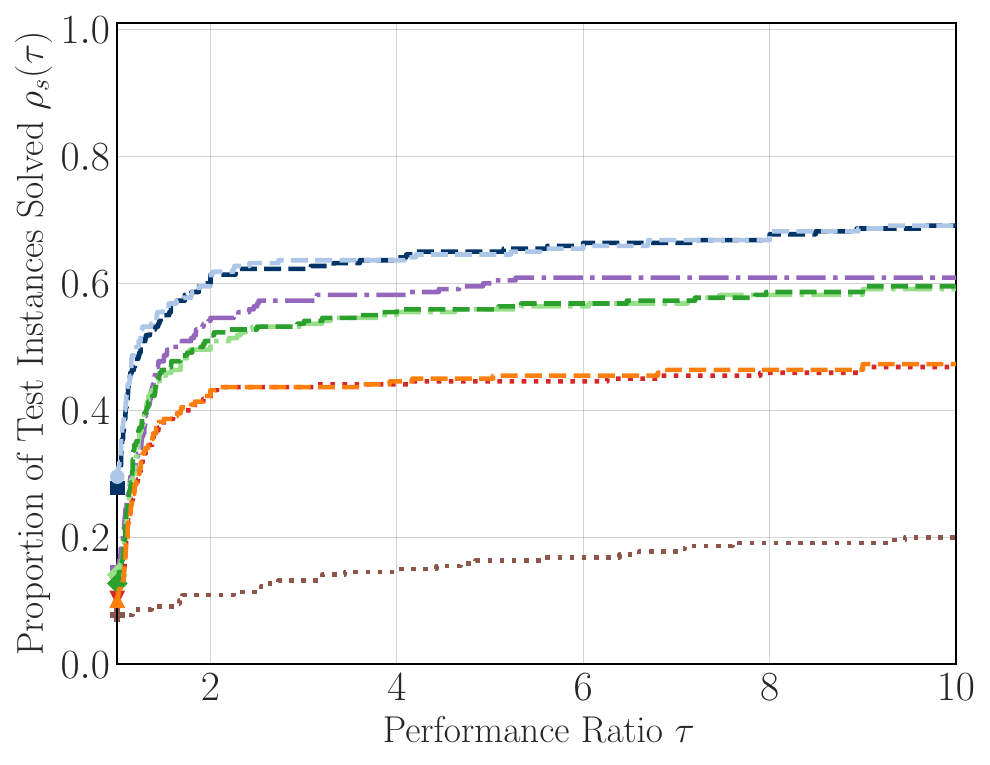}
        \caption{32-bit, $\gtol=10^{-3}$}
    \end{subfigure}
    \hfill
    \begin{subfigure}[b]{\resultsAllPrecisionsSubfigWidth\textwidth}
        \centering
        \includegraphics[width=\textwidth]{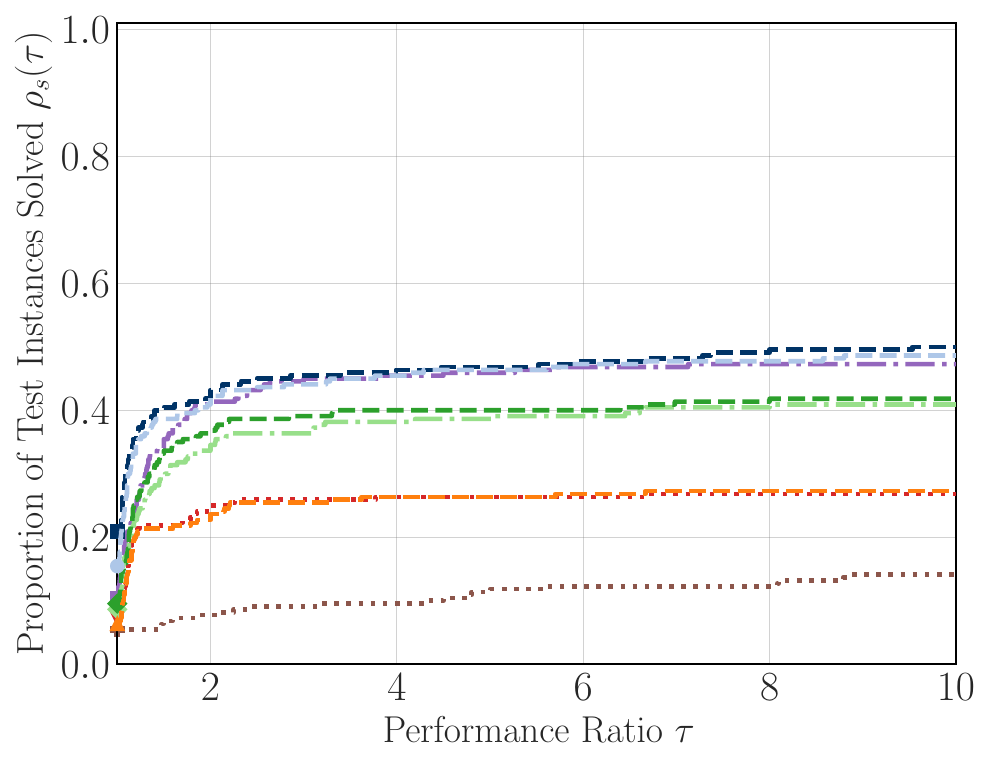}
        \caption{32-bit, $\gtol=10^{-5}$}
    \end{subfigure}

    \vspace{0.5em}

    \begin{subfigure}[b]{\resultsAllPrecisionsSubfigWidth\textwidth}
        \centering
        \includegraphics[width=\textwidth]{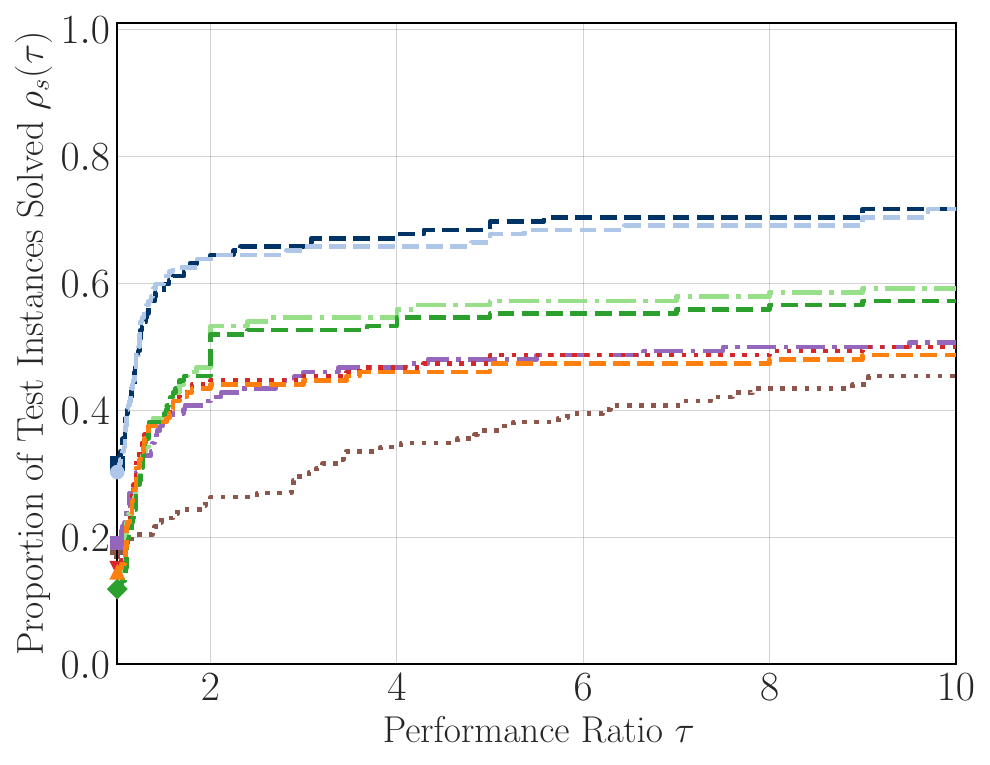}
        \caption{16-bit, $\gtol=10^{-1}$}
    \end{subfigure}
    \hfill
    \begin{subfigure}[b]{\resultsAllPrecisionsSubfigWidth\textwidth}
        \centering
        \includegraphics[width=\textwidth]{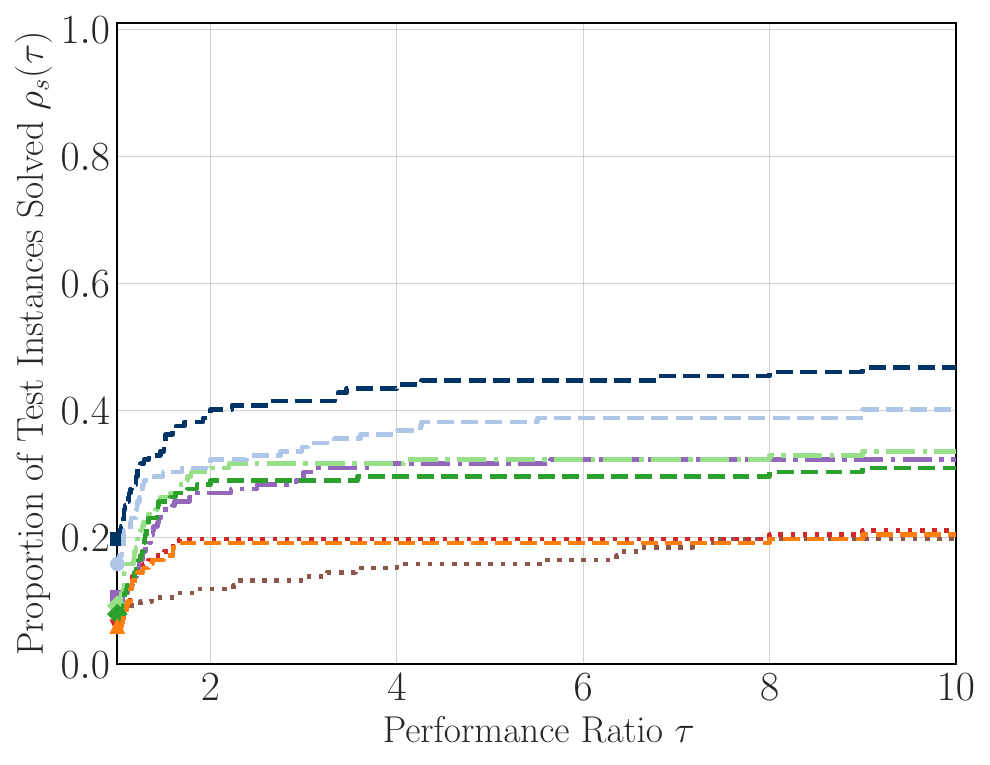}
        \caption{16-bit, $\gtol=10^{-3}$}
    \end{subfigure}
    \hfill
    \begin{subfigure}[b]{\resultsAllPrecisionsSubfigWidth\textwidth}
        \centering
        \includegraphics[width=\textwidth]{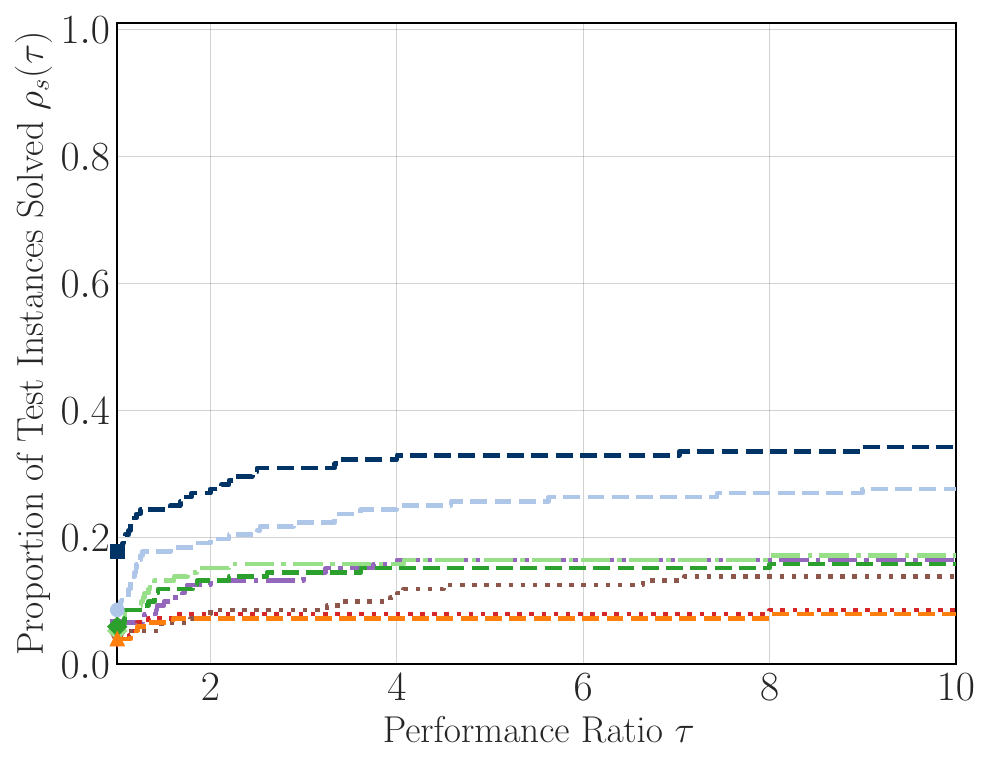}
        \caption{16-bit, $\gtol=10^{-5}$}
    \end{subfigure}

    \vspace{0.5em}

    \begin{subfigure}[b]{0.98\textwidth}
        \centering
        \includegraphics[width=\textwidth]{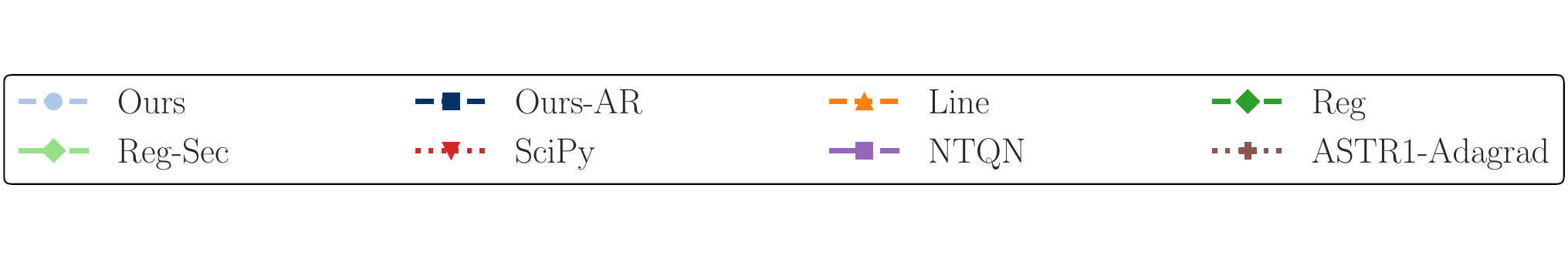}
    \end{subfigure}
    \caption{Performance profiles across evaluation precisions and gradient tolerances.}
    \label{fig:results_all_precisions}
\end{figure}

\subsection{Computational Time}\label{sec:comp_time}

We next compare computational cost from two complementary wall-clock perspectives.

\subsubsection{Per-Iteration Computational Cost}\label{sssec:per_iteration_cost}

\begin{figure}[t]
    \centering
    \includegraphics[width=0.9\textwidth]{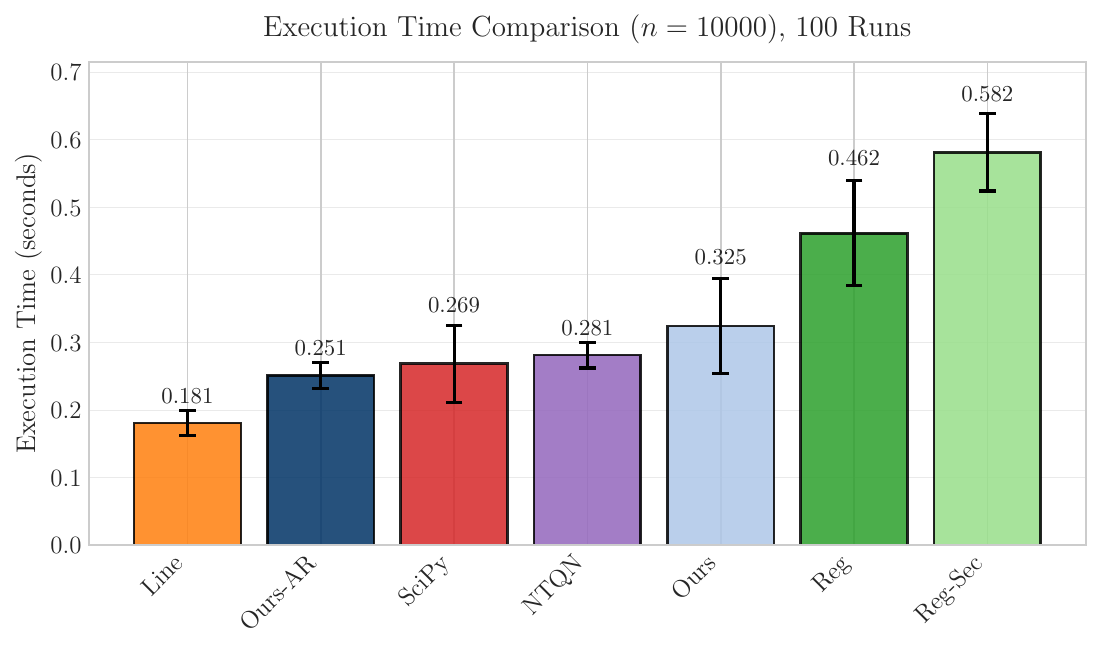}
    \caption{Mean execution time for $100$ iterations on the ill-conditioned quadratic problem.
        Error bars show one standard deviation over $100$ runs.}
    \label{fig:exec_time}
\end{figure}

The first experiment focuses on the per-iteration computational cost of each method.
\MPCchanged{We exclude \texttt{ASTR1-Adagrad} from this experiment because this first-order method uses gradient evaluations but no function evaluations and therefore has a different per-iteration workload from the quasi-Newton methods.}
We conducted experiments on the ill-conditioned quadratic problem
\begin{equation*}
    \underset{x \in \bbR^n}{\mathrm{minimize}} \quad \frac{1}{2} \sum_{i=1}^n i x_i^2
\end{equation*}
with $n=10{,}000$.
The initial point $x_0$ was chosen as the all-ones vector.
We used memory size $p=10$ and the iteration limit $k_{\max}=100$, with no other stopping criterion.
Because every method performs exactly $100$ iterations and the objective and gradient evaluations are inexpensive, differences in total runtime primarily reflect differences in per-iteration algorithmic overhead.
Timing data were collected after three warm-up runs, followed by $100$ recorded trials.

In \cref{fig:exec_time}, we report the execution-time statistics.
\texttt{Line} has the lowest per-iteration cost among the quasi-Newton methods.
\texttt{Ours-AR} is less expensive than \texttt{Ours}, while both proposed variants remain faster than the regularized L-BFGS baselines in this experiment.
This timing experiment measures algorithmic overhead under deliberately inexpensive objective and gradient evaluations and does not by itself establish end-to-end efficiency or superiority in oracle cost.

In \cref{tab:time_results}, we report the total number of function and gradient evaluations and final objective values $\overline{f}(x_{k_{\max}})$.
The comparable evaluation counts and final objective values indicate that differences in convergence behavior do not dominate the execution-time comparison.

\begin{table}[t]
    \centering
    \small
    \setlength{\tabcolsep}{3.5pt}
    \caption{The number of oracle calls and final observed objective values for the ill-conditioned quadratic experiment in \cref{sssec:per_iteration_cost}.}
    \label{tab:time_results}
    \begin{tabular}{lrrrrrrr}
    \toprule
                                & \texttt{Line} & \texttt{Ours-AR} & \texttt{SciPy} & \texttt{NTQN} & \texttt{Ours} & \texttt{Reg} & \texttt{Reg-Sec} \\
    \midrule
            \# Calls & 212 & 202 & 212 & 219 & 202 & 206 & 206 \\
    $\overline{f}(x_{k_{\max}})$ & 1.24 & 1.34 & 1.24 & 1.43 & 1.34 & 1.47 & 1.48 \\
    \bottomrule
    \end{tabular}
\end{table}

\subsubsection{End-to-End Time to Solution}

\begin{MPCchange}
Finally, we assess whether the higher solved percentages under reduced precision persist when performance is measured against elapsed time.
For this comparison, time to solution is the elapsed time until the first iterate attaining the common gradient tolerance, and a run that does not attain the tolerance is counted as unsolved.
\Cref{fig:time_to_solution} plots the fraction of 32-bit problems attaining $\gtol=10^{-3}$ as a function of time to solution.
At this tolerance, \texttt{Ours} solves $74.1\%$ of the problems and \texttt{Ours-AR} solves $73.2\%$, compared with $63.6\%$ for the best existing baseline.
\Cref{fig:time_to_solution} shows that the proposed variants also perform favorably against the existing baselines when performance is measured by elapsed time.
\end{MPCchange}

\begin{MPCchange}
\timeToSolutionFigure
    \ifdefined\MPCRevision\captionsetup{font={color=red}}\fi
    \centering
    \begin{subfigure}[b]{0.72\textwidth}
        \centering
        \includegraphics[width=\textwidth]{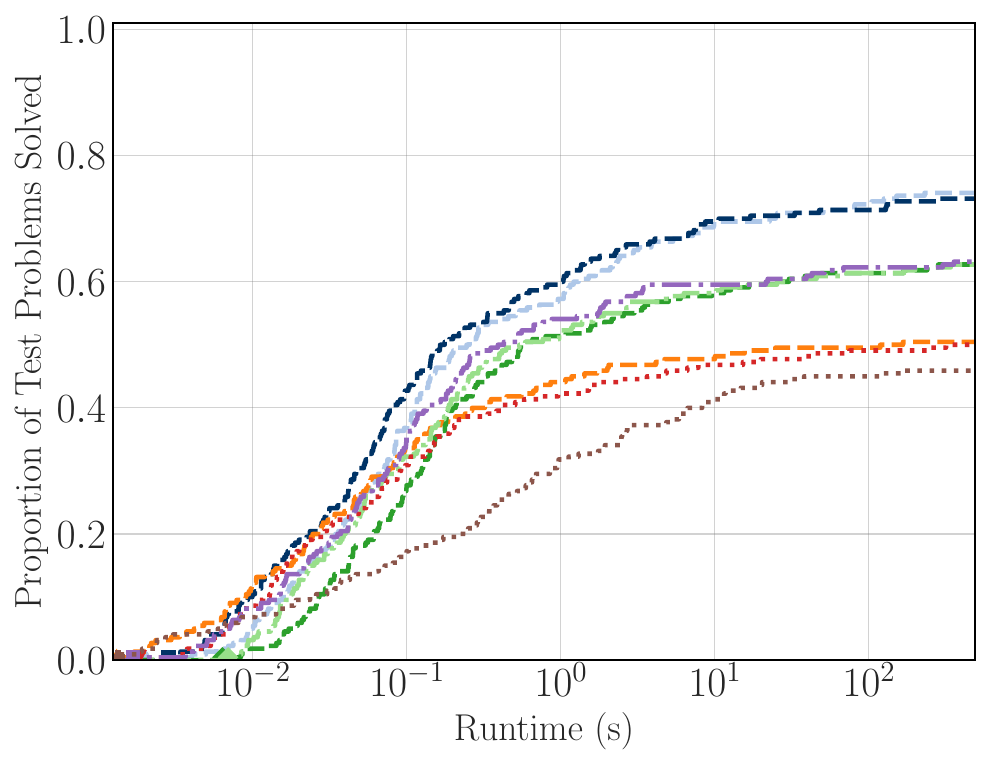}
    \end{subfigure}
    \begin{subfigure}[b]{0.98\textwidth}
        \centering
        \includegraphics[width=\textwidth]{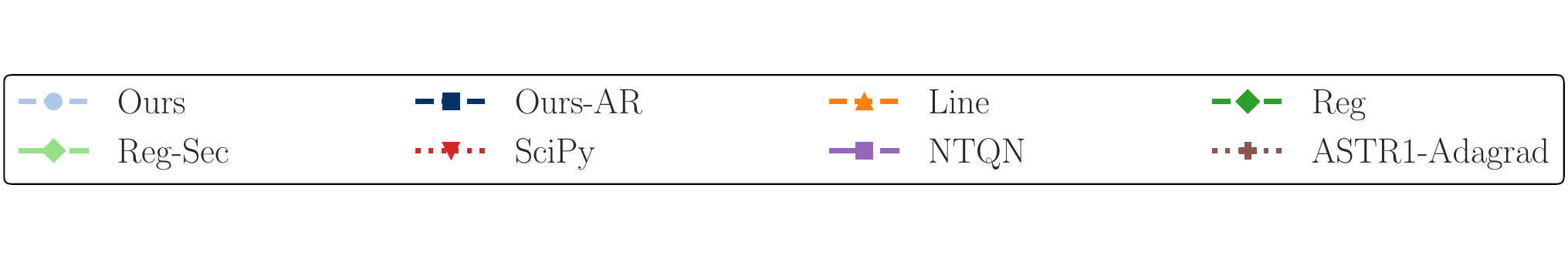}
    \end{subfigure}
    \caption{Time-to-solution comparison with 32-bit evaluations and $\gtol = 10^{-3}$.
Each curve shows the fraction of problems reaching the gradient tolerance within the elapsed time.}
    \label{fig:time_to_solution}
\end{figure}
\end{MPCchange}

\FloatBarrier
\section{Conclusion}\label{sec:conclusion}

We proposed a hybrid regularized quasi-Newton method for smooth nonconvex optimization with objective-function errors satisfying a hybrid absolute-relative model.
The method retains the \(\order{\varepsilon^{-2}}\) first-order iteration complexity of the $L$-smooth nonconvex setting while selectively using noisy function values.
At 16-bit and 32-bit precision, both proposed variants reached the strictest tested gradient tolerance on more CUTEst problems than the existing quasi-Newton baselines.

Future work includes the investigation of local convergence properties, the extension to constrained optimization, and the applications to practical problems in machine learning and scientific computing.
Additionally, exploring adaptive strategies for selecting algorithmic parameters based on problem characteristics could further enhance the method's performance.
An explicit analysis of the effect of inexact gradient evaluations would also provide valuable insight into the method's stability and guide further enhancements.

\section*{Acknowledgements}

This work was supported by JSPS KAKENHI Grant Numbers JP23K28041, JP24K23853, JP26KJ0936 and JST CREST (JPMJCR24Q2).


\begin{thebibliography}{10}
\providecommand{\url}[1]{{#1}}
\providecommand{\urlprefix}{URL }
\expandafter\ifx\csname urlstyle\endcsname\relax
  \providecommand{\doi}[1]{DOI~\discretionary{}{}{}#1}\else
  \providecommand{\doi}{DOI~\discretionary{}{}{}\begingroup
  \urlstyle{rm}\Url}\fi

\bibitem{bellaviaImpactNoiseEvaluation2021}
Bellavia, S., Gurioli, G., Morini, B., Toint, P.L.: The {{Impact}} of {{Noise}}
  on {{Evaluation Complexity}}: {{The Deterministic Trust-Region Case}}.
\newblock Journal of Optimization Theory and Applications \textbf{196}(2),
  700--729 (2023).
\newblock \doi{10.1007/s10957-022-02153-5}

\bibitem{berahasTheoreticalEmpiricalComparison2021}
Berahas, A.S., Cao, L., Choromanski, K., Scheinberg, K.: A {{Theoretical}} and
  {{Empirical Comparison}} of {{Gradient Approximations}} in {{Derivative-Free
  Optimization}}.
\newblock Foundations of Computational Mathematics \textbf{22}(2), 507--560
  (2022).
\newblock \doi{10.1007/s10208-021-09513-z}

\bibitem{berahasLimitedmemoryBFGSDisplacement2022}
Berahas, A.S., Curtis, F.E., Zhou, B.: Limited-memory {{BFGS}} with
  displacement aggregation.
\newblock Mathematical Programming \textbf{194}(1), 121--157 (2022).
\newblock \doi{10.1007/s10107-021-01621-6}

\bibitem{caflischMonteCarloQuasiMonte1998}
Caflisch, R.E.: Monte {{Carlo}} and quasi-{{Monte Carlo}} methods.
\newblock Acta Numerica \textbf{7}, 1--49 (1998).
\newblock \doi{10.1017/S0962492900002804}

\bibitem{carterNumericalExperienceClass1993}
Carter, R.G.: Numerical {{Experience}} with a {{Class}} of {{Algorithms}} for
  {{Nonlinear Optimization Using Inexact Function}} and {{Gradient
  Information}}.
\newblock SIAM Journal on Scientific Computing \textbf{14}(2), 368--388 (1993).
\newblock \doi{10.1137/0914023}

\bibitem{clancyTROPHYTrustRegion2022}
Clancy, R.J., Menickelly, M., H{\"u}ckelheim, J., Hovland, P., Nalluri, P.,
  Gjini, R.: {{TROPHY}}: {{Trust Region Optimization Using}} a {{Precision
  Hierarchy}}.
\newblock In: Computational {{Science}} -- {{ICCS}} 2022, pp. 445--459.
  Springer, Cham (2022).
\newblock \doi{10.1007/978-3-031-08751-6_32}

\bibitem{corlessLambertWFunction1996}
Corless, R.M., Gonnet, G.H., Hare, D.E.G., Jeffrey, D.J., Knuth, D.E.: On the
  {{Lambert W}} function.
\newblock Advances in Computational Mathematics \textbf{5}(1), 329--359 (1996).
\newblock \doi{10.1007/BF02124750}

\bibitem{doikovSuperUniversalRegularizedNewton2024}
Doikov, N., Mishchenko, K., Nesterov, Y.: Super-{{Universal Regularized Newton
  Method}}.
\newblock SIAM Journal on Optimization \textbf{34}(1), 27--56 (2024).
\newblock \doi{10.1137/22M1519444}

\bibitem{dolanBenchmarkingOptimizationSoftware2002}
Dolan, E.D., Mor{\'e}, J.J.: Benchmarking optimization software with
  performance profiles.
\newblock Mathematical Programming \textbf{91}(2), 201--213 (2002).
\newblock \doi{10.1007/s101070100263}

\bibitem{duchiAdaptiveSubgradientMethods2011}
Duchi, J., Hazan, E., Singer, Y.: Adaptive {{Subgradient Methods}} for {{Online
  Learning}} and {{Stochastic Optimization}}.
\newblock Journal of Machine Learning Research \textbf{12}(61), 2121--2159
  (2011).
\newblock \urlprefix\url{http://jmlr.org/papers/v12/duchi11a.html}

\bibitem{PyCUTEst2022}
Fowkes, J., Roberts, L., B{\H u}rmen, {\'A}.: {{PyCUTEst}}: An open source
  {{Python}} package of optimization test problems.
\newblock Journal of Open Source Software \textbf{7}(78), 4377 (2022).
\newblock \doi{10.21105/joss.04377}

\bibitem{gaoQuasiNewtonMethodsSuperlinear2018}
Gao, W., Goldfarb, D.: Quasi-{{Newton}} methods: Superlinear convergence
  without line searches for self-concordant functions.
\newblock Optimization Methods and Software \textbf{34}(1), 194--217 (2019).
\newblock \doi{10.1080/10556788.2018.1510927}

\bibitem{gouldCUTEstConstrainedUnconstrained2015}
Gould, N.I., Orban, D., Toint, P.L.: {{CUTEst}}: A {{Constrained}} and
  {{Unconstrained Testing Environment}} with safe threads for mathematical
  optimization.
\newblock Computational Optimization and Applications \textbf{60}(3), 545--557
  (2015).
\newblock \doi{10.1007/s10589-014-9687-3}

\bibitem{gowerStochasticBlockBFGS2016}
Gower, R., Goldfarb, D., Richtarik, P.: Stochastic {{Block BFGS}}: {{Squeezing
  More Curvature}} out of {{Data}}.
\newblock In: Proceedings of {{The}} 33rd {{International Conference}} on
  {{Machine Learning}}, pp. 1869--1878. PMLR (2016).
\newblock \urlprefix\url{https://proceedings.mlr.press/v48/gower16.html}

\bibitem{grapigliaAdaptiveTrustregionMethod2022}
Grapiglia, G.N., Stella, G.F.D.: An adaptive trust-region method without
  function evaluations.
\newblock Computational Optimization and Applications \textbf{82}(1), 31--60
  (2022).
\newblock \doi{10.1007/s10589-022-00356-0}

\bibitem{Gratton08022024}
Gratton, S., Jerad, S., Toint, {\relax Ph}.L.: Complexity of a class of
  first-order objective-function-free optimization algorithms.
\newblock Optimization Methods and Software \textbf{41}(2), 478--508 (2026).
\newblock \doi{10.1080/10556788.2023.2296431}

\bibitem{grattonNoteSolvingNonlinear2020}
Gratton, S., Toint, P.L.: A note on solving nonlinear optimization problems in
  variable precision.
\newblock Computational Optimization and Applications \textbf{76}(3), 917--933
  (2020).
\newblock \doi{10.1007/s10589-020-00190-2}

\bibitem{grattonOFFOMinimizationAlgorithms2023}
Gratton, S., Toint, P.L.: {{OFFO}} minimization algorithms for second-order
  optimality and their complexity.
\newblock Computational Optimization and Applications \textbf{84}(2), 573--607
  (2023).
\newblock \doi{10.1007/s10589-022-00435-2}

\bibitem{highamAccuracyStabilityNumerical2002}
Higham, N.J.: Accuracy and {{Stability}} of {{Numerical Algorithms}}, 2 edn.
\newblock {Society for Industrial and Applied Mathematics} (2002).
\newblock \doi{10.1137/1.9780898718027}

\bibitem{10.5555/3168451.3168462}
Izycheva, A., Darulova, E.: On sound relative error bounds for floating-point
  arithmetic.
\newblock In: Proceedings of the 17th Conference on Formal Methods in
  Computer-Aided Design, pp. 15--22. FMCAD Inc, Austin, Texas (2017).
\newblock \doi{10.23919/FMCAD.2017.8102236}

\bibitem{kanzowRegularizationLimitedMemory2023}
Kanzow, C., Steck, D.: Regularization of limited memory quasi-{{Newton}}
  methods for large-scale nonconvex minimization.
\newblock Mathematical Programming Computation \textbf{15}(3), 417--444 (2023).
\newblock \doi{10.1007/s12532-023-00238-4}

\bibitem{liGlobalConvergenceBFGS2001}
Li, D.H., Fukushima, M.: On the {{Global Convergence}} of the {{BFGS Method}}
  for {{Nonconvex Unconstrained Optimization Problems}}.
\newblock SIAM Journal on Optimization \textbf{11}(4), 1054--1064 (2001).
\newblock \doi{10.1137/S1052623499354242}

\bibitem{liTruncatedRegularizedNewton2009}
Li, Y.J., Li, D.H.: Truncated regularized {{Newton}} method for convex
  minimizations.
\newblock Computational Optimization and Applications \textbf{43}(1), 119--131
  (2009).
\newblock \doi{10.1007/s10589-007-9128-7}

\bibitem{liuLimitedMemoryBFGS1989}
Liu, D.C., Nocedal, J.: On the limited memory {{BFGS}} method for large scale
  optimization.
\newblock Mathematical Programming \textbf{45}(1), 503--528 (1989).
\newblock \doi{10.1007/BF01589116}

\bibitem{lotfiStochasticDampedLBFGS2020}
Lotfi, S., Bonniot, T., Orban, D., Lodi, A.: Stochastic damped {{L-BFGS}} with
  controlled norm of the {{Hessian}} approximation.
\newblock Les {{Cahiers}} Du {{GERAD}} G-2020-52, Groupe d'\'etudes et de
  recherche en analyse des d\'ecisions, GERAD, Montr\'eal QC H3T 2A7, Canada
  (2020).
\newblock \urlprefix\url{https://www.gerad.ca/en/papers/G-2020-52}

\bibitem{mannelGlobalizationLBFGSBarzilai2025}
Mannel, F.: A {{Globalization}} of {{L-BFGS}}~and the
  {{Barzilai}}--{{Borwein~Method}} for {{Nonconvex Unconstrained
  Optimization}}.
\newblock Journal of Optimization Theory and Applications \textbf{204}(3), 47
  (2025).
\newblock \doi{10.1007/s10957-024-02565-5}

\bibitem{mannelStructuredLBFGSMethod2024}
Mannel, F., Om~Aggrawal, H., Modersitzki, J.: A structured {{L-BFGS}} method
  and its application to inverse problems.
\newblock Inverse Problems \textbf{40}(4), 045022 (2024).
\newblock \doi{10.1088/1361-6420/ad2c31}

\bibitem{10.1145/192115.192132}
Mor{\'e}, J.J., Thuente, D.J.: Line search algorithms with guaranteed
  sufficient decrease.
\newblock ACM Trans. Math. Softw. \textbf{20}(3), 286--307 (1994).
\newblock \doi{10.1145/192115.192132}

\bibitem{moreEstimatingComputationalNoise2011}
Mor{\'e}, J.J., Wild, S.M.: Estimating {{Computational Noise}}.
\newblock SIAM Journal on Scientific Computing \textbf{33}(3), 1292--1314
  (2011).
\newblock \doi{10.1137/100786125}

\bibitem{nesterovCubicRegularizationNewton2006}
Nesterov, Y., Polyak, B.: Cubic regularization of {{Newton}} method and its
  global performance.
\newblock Mathematical Programming \textbf{108}(1), 177--205 (2006).
\newblock \doi{10.1007/s10107-006-0706-8}

\bibitem{nocedal1999numerical}
Nocedal, J., Wright, S.J.: Numerical Optimization.
\newblock Springer (1999).
\newblock \doi{10.1007/978-0-387-40065-5}

\bibitem{Okazaki2007libLBFGS}
Okazaki, N.: {{libLBFGS}}: A library of limited-memory
  broyden-fletcher-goldfarb-shanno (l-{{BFGS}}) (2007).
\newblock \urlprefix\url{https://github.com/chokkan/liblbfgs}

\bibitem{powellAlgorithmsNonlinearConstraints1978}
Powell, M.J.: Algorithms for nonlinear constraints that use lagrangian
  functions.
\newblock Math. Program. \textbf{14}(1), 224--248 (1978).
\newblock \doi{10.1007/BF01588967}

\bibitem{ranganathSymmetricRankOneQuasiNewton2025}
Ranganath, A., Singhal, M., Marcia, R.: Symmetric {{Rank-One Quasi-Newton
  Methods}} for {{Deep Learning Using Cubic Regularization}}.
\newblock In: 2025 {{IEEE}} 10th {{International Workshop}} on {{Computational
  Advances}} in {{Multi-Sensor Adaptive Processing}} ({{CAMSAP}}), pp. 121--125
  (2025).
\newblock \doi{10.1109/CAMSAP66162.2025.11423871}

\bibitem{ruanAdaptiveRegularizedQuasiNewton2024}
Ruan, H., Yang, W.: Adaptive {{Regularized Quasi-Newton Method Using Inexact
  First-Order Information}}.
\newblock Journal of Computational Mathematics \textbf{42}(6), 1656--1687
  (2024).
\newblock \doi{10.4208/jcm.2306-m2022-0279}

\bibitem{shengNewAdaptiveTrust2018}
Sheng, Z., Yuan, G., Cui, Z.: A new adaptive trust region algorithm for
  optimization problems.
\newblock Acta Mathematica Scientia \textbf{38}(2), 479--496 (2018).
\newblock \doi{10.1016/S0252-9602(18)30762-8}

\bibitem{shiNoiseTolerantQuasiNewtonAlgorithm2022}
Shi, H.J.M., Xie, Y., Byrd, R., Nocedal, J.: A {{Noise-Tolerant Quasi-Newton
  Algorithm}} for {{Unconstrained Optimization}}.
\newblock SIAM Journal on Optimization \textbf{32}(1), 29--55 (2022).
\newblock \doi{10.1137/20M1373190}

\bibitem{sunTrustRegionMethod2023}
Sun, S., Nocedal, J.: A trust region method for noisy unconstrained
  optimization.
\newblock Mathematical Programming \textbf{202}(1), 445--472 (2023).
\newblock \doi{10.1007/s10107-023-01941-9}

\bibitem{uedaConvergencePropertiesRegularized2010}
Ueda, K., Yamashita, N.: Convergence {{Properties}} of the {{Regularized Newton
  Method}} for~the~{{Unconstrained Nonconvex Optimization}}.
\newblock Applied Mathematics and Optimization \textbf{62}(1), 27--46 (2010).
\newblock \doi{10.1007/s00245-009-9094-9}

\bibitem{uedaRegularizedNewtonMethod2014}
Ueda, K., Yamashita, N.: A regularized {{Newton}} method without line search
  for unconstrained optimization.
\newblock Computational Optimization and Applications \textbf{59}(1), 321--351
  (2014).
\newblock \doi{10.1007/s10589-014-9656-x}

\bibitem{2020SciPy-NMeth}
Virtanen, P., Gommers, R., Oliphant, T.E., Haberland, M., Reddy, T.,
  Cournapeau, D., Burovski, E., Peterson, P., Weckesser, W., Bright, J., {van
  der Walt}, S.J., Brett, M., Wilson, J., Millman, K.J., Mayorov, N., Nelson,
  A.R.J., Jones, E., Kern, R., Larson, E., Carey, C.J., Polat, I., Feng, Y.,
  Moore, E.W., VanderPlas, J., Laxalde, D., Perktold, J., Cimrman, R.,
  Henriksen, I., Quintero, E.A., Harris, C.R., Archibald, A.M., Ribeiro, A.H.,
  Pedregosa, F., {van Mulbregt}, P., {SciPy 1.0 Contributors}: {{SciPy}} 1.0:
  {{Fundamental}} algorithms for scientific computing in python.
\newblock Nature Methods \textbf{17}, 261--272 (2020).
\newblock \doi{10.1038/s41592-019-0686-2}

\bibitem{wangConvergenceRatesRegularized2025}
Wang, S., Fadili, J., Ochs, P.: Convergence rates of regularized
  quasi-{{Newton}} methods without strong convexity.
\newblock arXiv preprint arXiv:2506.00521  (2025).
\newblock \doi{10.48550/arXiv.2506.00521}

\bibitem{wardAdaGradStepsizesSharp2020}
Ward, R., Wu, X., Bottou, L.: {{AdaGrad}} stepsizes: {{Sharp}} convergence over
  nonconvex landscapes.
\newblock Journal of Machine Learning Research \textbf{21}(219), 1--30 (2020).
\newblock \urlprefix\url{http://jmlr.org/papers/v21/18-352.html}

\bibitem{xieAnalysisBFGSMethod2020}
Xie, Y., Byrd, R.H., Nocedal, J.: Analysis of the {{BFGS Method}} with
  {{Errors}}.
\newblock SIAM Journal on Optimization \textbf{30}(1), 182--209 (2020).
\newblock \doi{10.1137/19M1240794}

\bibitem{yuanConvergenceAnalysisModified2010}
Yuan, G., Wei, Z.: Convergence analysis of a modified {{BFGS}} method on~convex
  minimizations.
\newblock Computational Optimization and Applications \textbf{47}(2), 237--255
  (2010).
\newblock \doi{10.1007/s10589-008-9219-0}

\bibitem{zhangNewRegularizedQuasiNewton2018}
Zhang, H., Ni, Q.: A new regularized quasi-{{Newton}} method for unconstrained
  optimization.
\newblock Optimization Letters \textbf{12}(7), 1639--1658 (2018).
\newblock \doi{10.1007/s11590-018-1236-z}

\bibitem{zhangNewQuasiNewtonEquation1999}
Zhang, J.Z., Deng, N.Y., Chen, L.H.: New {{Quasi-Newton Equation}} and
  {{Related Methods}} for {{Unconstrained Optimization}}.
\newblock Journal of Optimization Theory and Applications \textbf{102}(1),
  147--167 (1999).
\newblock \doi{10.1023/A:1021898630001}

\end{thebibliography}

\begin{appendices}
    \crefalias{section}{appendix}
    \section{Proof of \texorpdfstring{\cref{lem:sum_norm_g_over_mu_geq_2sqrt_leq_log}}{Lemma \ref{lem:sum_norm_g_over_mu_geq_2sqrt_leq_log}}}\label{sec:proof_of_sum_norm_g_over_mu_geq_2sqrt_leq_log}

The proof of \cref{lem:sum_norm_g_over_mu_geq_2sqrt_leq_log} is similar to previous arguments~\citep[Lemma~3.2]{wardAdaGradStepsizesSharp2020} and \citep[Lemma~3.1]{Gratton08022024}.
\begin{proof}
    Recall that \cref{eq:offo_based_update} defines
    \begin{equation*}
        \mu_k=\theta_k \sqrt{\varsigma + \sum_{j \in K^+_{k+1}} \norm{g_j}^2}.
    \end{equation*}
    Moreover, \refLine{line:update_theta} of \cref{alg:regularized_qn} asserts $\theta_k \in [\theta_{\min}, \theta_{\max}]$.
    Let $k_i \in K^+$ be the $i$-th element of $K^+$.
    Define $S_{-1} \coloneqq \varsigma$ and $S_i \coloneqq \varsigma + \sum_{\ell=0}^{i} \norm{g_{k_\ell}}^2$ for $i \geq 0$.
    For any $0 < x < y$, $x/y \le -\ln(1-x/y) = \ln y - \ln(y-x)$ holds.
    Thus, we have
    \begin{align*}
        \frac{\norm{g_{k_i}}^2}{\mu_{k_i}^2}
         & \leq \frac{1}{\theta_{\min}^2} \frac{\norm{g_{k_i}}^2}{S_i} &  & (\text{definition}) \\
         & \leq \frac{1}{\theta_{\min}^2} \qty(\ln S_i - \ln S_{i-1}).  &  & (x/y \leq \ln y - \ln(y-x))
    \end{align*}
    Similarly, $x/\sqrt{y} \ge x/(\sqrt{y}+\sqrt{y-x}) = \sqrt{y}-\sqrt{y-x}$ holds.
    Thus, we have
    \begin{align*}
        \frac{\norm{g_{k_i}}^2}{\mu_{k_i}}
         & \geq \frac{1}{\theta_{\max}} \frac{\norm{g_{k_i}}^2}{\sqrt{S_i}} &  & (\text{definition}) \\
         & \geq \frac{1}{\theta_{\max}} \qty(\sqrt{S_i} - \sqrt{S_{i-1}}). &  & (x/\sqrt{y} \geq \sqrt{y} - \sqrt{y-x})
    \end{align*}
    Summing these inequalities over $0 \le i < \abs{K^+}$ yields the claim.
    \myQED
\end{proof}

\section{Proof of \texorpdfstring{\cref{prop:bounded_BFGS}}{Proposition \ref{prop:bounded_BFGS}}}\label{app:proof_of_bounded_BFGS}

The proof of \cref{prop:bounded_BFGS} uses the standard BFGS update.
\begin{proof}
    From \cref{eq:curvature_conditions1}, the initial scalar satisfies $0<\gamma=\norm{y_{k_{p-1}}}^2/(y_{k_{p-1}}^\top s_{k_{p-1}})\leq\Lambda$.
    Let $B$ and $B_+$ denote the positive definite Hessian approximations before and after one update.
    The BFGS update rule is
    \begin{equation*}
        B_+=B - \frac{Bss^\top B}{s^\top Bs} + \frac{yy^\top}{y^\top s}.
    \end{equation*}
    Positive definiteness follows from $B\succ0$, $s\neq0$, and $y^\top s>0$.
    Moreover, \cref{eq:curvature_conditions1} and the fact that $Bss^\top B/(s^\top Bs)$ is positive semidefinite give
    \begin{equation*}
        B_+\preceq B+\frac{yy^\top}{y^\top s}\preceq B+\Lambda I.
    \end{equation*}
    Applying this inequality recursively over $p$ updates from $B^{(0)}=\gamma I\preceq\Lambda I$ proves \cref{eq:bounded_BFGS_mM}.
    \myQED
\end{proof}

\end{appendices}

\end{document}